\documentclass{amsart}

\usepackage{amssymb}
\usepackage[hidelinks]{hyperref}
\usepackage[textsize=footnotesize]{todonotes}
\usepackage{tikz}
\usetikzlibrary{cd}
\usepackage{float}
\usepackage{xcolor}
\usepackage{todonotes}
\usepackage{framed}

\theoremstyle{plain}
\newtheorem{theorem}{Theorem}[section]
\newtheorem{lemma}[theorem]{Lemma}
\newtheorem{corollary}[theorem]{Corollary}
\newtheorem{proposition}[theorem]{Proposition}

\theoremstyle{definition}
\newtheorem{definition}[theorem]{Definition}

\newtheorem{example}[theorem]{Example}

\newtheorem{question}[theorem]{Question}

\theoremstyle{remark}
\newtheorem*{remark}{Remark}

\newtheorem*{notation}{Notation}

\newcommand{\bR}{\mathbb{R}}
\newcommand{\R}{\bR}
\newcommand{\bQ}{\mathbb{Q}}
\newcommand{\Q}{\bQ}

\newcommand{\cA}{\mathcal{A}}
\newcommand{\cB}{\mathcal{B}}
\newcommand{\cC}{\mathcal{C}}

\newcommand{\cD}{\mathcal{D}}
\newcommand{\cE}{\mathcal{E}}
\newcommand{\cF}{\mathcal{F}}

\newcommand{\cI}{\mathcal{I}}
\newcommand{\I}{\cI}
\newcommand{\cJ}{\mathcal{J}}
\newcommand{\J}{\cJ}
\newcommand{\cK}{\mathcal{K}}
\newcommand{\K}{\cK}
\newcommand{\cL}{\mathcal{L}}
\newcommand{\cM}{\mathcal{M}}
\newcommand{\cN}{\mathcal{N}}
\newcommand{\cP}{\mathcal{P}}

\newcommand{\cS}{\mathcal{S}}

\newcommand{\continuum}{\mathfrak{c}}
\newcommand{\cc}{\continuum}
\newcommand{\bnumber}{\mathfrak{b}}
\newcommand{\bb}{\bnumber}
\newcommand{\bs}{\bnumber_\sigma}

\newcommand{\cf}{\mathrm{cf}} 
\DeclareMathOperator{\add}{add}
\DeclareMathOperator{\non}{non}
\DeclareMathOperator{\adds}{add^\star}
\newcommand{\fin}{\mathrm{Fin}}
\newcommand{\Fin}{\mathrm{Fin}}

\DeclareMathOperator{\chf}{\mathbf{1}}
\DeclareMathOperator{\dom}{dom}
\DeclareMathOperator{\ran}{ran}

\makeatletter
\def\arrowfill@@#1#2#3#4{%
  $\m@th\thickmuskip0mu\medmuskip\thickmuskip\thinmuskip\thickmuskip
   \relax#4#1
   \xleaders\hbox{$#4#2$}\hfill
   #3$%
}
\def\dashrightarrowfill@@{\arrowfill@@\relax\relbar\rightarrow}
\def\dotrightarrowfill@@{\arrowfill@@\relax\cdot\rightarrow}
\newcommand{\xdashrightarrow}[2][]{\ext@arrow 0359\dashrightarrowfill@@{#1}{#2}}
\newcommand{\xdotrightarrow}[2][]{\ext@arrow 0359\dotrightarrowfill@@{#1}{#2}}
\makeatother

\begin{document}


\title{Spaces not distinguishing ideal pointwise and $\sigma$-uniform convergence}


\author{Rafa\l{} Filip\'{o}w}
\address[Rafa\l{}~Filip\'{o}w]{Institute of Mathematics\\ Faculty of Mathematics, Physics and Informatics\\ University of Gda\'{n}sk\\ ul.~Wita Stwosza 57\\ 80-308 Gda\'{n}sk\\ Poland}
\email{Rafal.Filipow@ug.edu.pl}
\urladdr{http://mat.ug.edu.pl/~rfilipow}

\author[Adam Kwela]{Adam Kwela}
\address[Adam Kwela]{Institute of Mathematics\\ Faculty of Mathematics\\ Physics and Informatics\\ University of Gda\'{n}sk\\ ul.~Wita  Stwosza 57\\ 80-308 Gda\'{n}sk\\ Poland}
\email{Adam.Kwela@ug.edu.pl}
\urladdr{http://mat.ug.edu.pl/~akwela}


\date{\today}






\begin{abstract}
We examine topological  spaces not distinguishing ideal pointwise and ideal $\sigma$-uniform convergence of sequences of real-valued continuous functions defined on them. 
For instance, we introduce a purely combinatorial cardinal characteristic (a sort of the bounding number $\bnumber$) and prove that it  describes the minimal cardinality of  topological spaces which  distinguish  ideal pointwise and ideal $\sigma$-uniform convergence. Moreover, we provide examples of topological spaces (focusing on subsets of reals) that do or do not distinguish  the considered  convergences. Since similar investigations for ideal quasi-normal convergence instead of ideal $\sigma$-uniform convergence have been performed in literature, we also study spaces not distinguishing ideal quasi-normal and ideal $\sigma$-uniform convergence of sequences of real-valued continuous functions defined on them. 
\end{abstract}


\maketitle

\setcounter{tocdepth}{1}
\tableofcontents


\section{Introduction}

A topological space $X$ is a \emph{QN-space} if it does not distinguish pointwise and quasi-normal convergence of sequences of real-valued continuous functions defined on $X$ (for  the definition of quasi-normal convergence and definitions of other notions used in Introduction see Section \ref{sec:prelim}).
QN-spaces were introduced by Bukovsk\'{y}, Rec\l{}aw and Repick\'{y} \cite{MR1129696} and were  thoroughly examined in the following years \cite{MR2463820,MR2778559,MR2294632,MR1129696,MR1815270,MR1477547,MR1800160,MR2280899,MR2881299}.

A notion of convergence (such as pointwise or quasi-normal convergence of sequences of functions) often can be generalized using ideals on the set of natural numbers. For instance, the ordinary convergence of sequences of reals generalized with the aid of the  ideal of sets of asymptotic density zero is known as the statistical convergence \cite{MR0048548,MR0816582,Steinhaus}.

    It is known \cite[Theorem~5.1]{MR0515120} (see also \cite[Theorem~1.2]{MR1108577}) that  quasi-normal convergence is equivalent to $\sigma$-uniform convergence. Thus, QN-spaces are in fact topological spaces not distinguishing pointwise and $\sigma$-uniform convergence of sequences of real-valued continuous functions defined on them.

The research on ideal analogues of QN-spaces, initiated by Das and Chandra~\cite{MR3038073} and continued by others \cite{MR3622377,MR3784399,MR4336563,MR4336564,MR3924519,MR3423409}, has concentrated only on spaces not distinguishing ideal pointwise and ideal quasi-normal convergence of sequences of continuous functions so far.
However, it is known \cite{MR3624786} that ideal quasi-normal and ideal $\sigma$-uniform convergence are not the same for a large class of ideals. What is more, $\sigma$-uniform convergence seems to be better known than quasi-normal convergence and ideal analogue of $\sigma$-uniform convergence seems more  natural than  ideal analogue of quasi-normal convergence (the latter was even initially introduced in two different ways \cite{MR3038073,MR2899832}). 

It seems that the research on ideal QN-spaces would be incomplete without studying spaces not distinguish ideal pointwise and ideal $\sigma$-uniform convergence of sequences of real-valued continuous functions defined on them. 
Our paper is an attempt to fill this gap, and it is organized in the following way.

In Section~\ref{sec:uniform-convergence}, we show (Corollary~\ref{cor:pointwise-versus-uniform}) that every infinite space distinguishes between ideal uniform convergence and the other considered convergences (i.e.~pointwise, $\sigma$-uniform and quasi-normal).
Moreover, we show (Corollary~\ref{cor:only-one-class}) that a space does not distinguish  ideal pointwise and $\sigma$-uniform convergence if and only if it simultaneously does not 
distinguish ideal pointwise and quasi-normal convergence and does not distinguish ideal quasi-normal and $\sigma$-uniform convergence.

In Section~\ref{sec:sigma-uniform}, we prove the main result of the paper (Corollary~\ref{cor:pointwise-versus-quasinormal:non}) which provides a purely combinatorial characterization of the  
minimal cardinality of a topological space which distinguishes ideal pointwise and ideal $\sigma$-uniform convergence of sequences of continuous functions.

In Section~\ref{sec:properties-of-bsigma}, we examine various properties of combinatorial cardinal characteristics introduced in the preceding section (some of these properties are used in the following sections).

In Section~\ref{sec:spaces-of-arbitrary-cardinality}, 
we show (Corollary~\ref{cor:spaces-of-arbitrary-cardinality-may-distinguis-convergences}) that the property of ``not distinguishing ideal pointwise and $\sigma$-uniform convergence of continuous functions'' is  of the topological nature rather than set-theoretic. We also provide (Theorem~\ref{thm:Sierpinski-set-not-distinguishes-convergence}) under CH an example of an uncountable  subspace of the reals revealing the above phenomenon.

In Section~\ref{sec:b-of-relations}, we show (Theorem~\ref{thm:b-sigma-as-Vojtas-b}) that combinatorial cardinal characteristics introduced in the preceding section can be described in a uniform manner as the bounding numbers of binary relations. These descriptions are crucial for the results obtain in the following section.

In Section~\ref{sec:subsets-of-R-distinguishing}, we construct (Theorem~\ref{thm:subset-of-R-not-distinguishing-convergence}) a subset of the reals of the minimal size  which distinguish the ideal pointwise convergence and $\sigma$-uniform convergence.

Finally in Section~\ref{sec:Distinguishing-between-spaces}, we show (Proposition~\ref{prop:distinguishing-between-spaces-not-distinguishing-convergence}) 
that consistently there exists  a space which does not distinguish ordinary pointwise convergence and ordinary $\sigma$-uniform convergence but it does distinguish statistical pointwise convergence and statistical $\sigma$-uniform convergence.


\section{Preliminaries}
\label{sec:prelim}

By $\omega$ we denote the set of all natural numbers.
We identify  a natural number $n$ with the set $\{0, 1,\dots , n-1\}$. 
We write $A\subseteq^*B$ if $A\setminus B$ is finite.
For a set $A$ and a cardinal number  $\kappa$, we write $[A]^{\kappa} =\{B\subseteq A: |B|=\kappa\}$, where $|B|$ denotes the cardinality of $B$. 

If $A$ and $B$ are two sets then by $A^B$ we denote the family of all functions $f:B\to A$. If $f\in A^B$ and $C\subseteq B$ then $f\restriction C:C\to A$ is the restriction of $f$ to $C$. In the case $B=\omega$, an element of $A^\omega$ will sometimes be denoted $(  a_n) $ -- by this we mean $f:\omega\to A$ given by $f(n)=a_n$ for all $n$.  

 For $A\subseteq X$, we write $\chf_{A}(n)$ to denote the characteristic function of $A$ i.e.~$\chf_A(x)=1$ for $x\in A$ and $\chf_A(x)=0$ for $x\in X\setminus A$.

By $\omega$, $\omega_1$ and $\cc$ we denote the first infinite cardinal, the first uncountable cardinal and the cardinality of $\mathbb{R}$, respectively. By $\cf(\kappa)$ we denote the cofinality of a  cardinal $\kappa$.


\subsection{Ideals}

An \emph{ideal on a set $X$} is a family $\I\subseteq\cP(X)$ that satisfies the following properties:
\begin{enumerate}
\item if $A,B\in \I$ then $A\cup B\in\I$,
\item if $A\subseteq B$ and $B\in\I$ then $A\in\I$,
\item $\I$ contains all finite subsets of $X$,
\item $X\notin\I$.
\end{enumerate}

An ideal $\I$ on $X$ is \emph{tall} if for every infinite $A\subseteq X$ there is an infinite $B\in\I$ such that $B\subset A$. An ideal $\I$ on $X$ is a \emph{P-ideal} if for any countable family $\cA\subseteq\I$ there is $B\in \I$ such that $A\setminus B$ is finite for every $A\in \cA$. 
An ideal $\I$ on $X$ is \emph{countably generated} if there is a countable family $\cB\subseteq \I$ such that for every $A\in \I$ there is $B\in \cB$ with $A\subseteq B$.

The vertical section of a set   $A\subseteq X\times Y$ at a point $x\in X$ is defined 
by
$(A)_x = \{y\in Y : (x,y)\in A\}$. 

For ideals $\I,\J$ on $X$ and $Y$, respectively, we define the following new ideals:
\begin{enumerate}
	\item
	$\I\otimes \J = \{A\subseteq X\times Y: \{x\in X: (A)_x\notin\J\}\in \I\}$,
	\item
$\I\otimes \{\emptyset\} = \{A\subseteq X\times \omega: \{x\in X: (A)_x\ne\emptyset\}\in \I\}$.
	\item
$\{\emptyset\} \otimes  \J= \{A\subseteq \omega\times Y : (A)_x\in \J\text{ for all }x\in X\}$.
\end{enumerate}

The following specific ideals will be considered in the paper (see e.g.~\cite{MR2777744} for these and many more examples).
\begin{example}\
\begin{itemize} 
\item $\fin=\{A\subseteq \omega: |A|<\omega\}$ is the ideal of all finite subsets of $\omega$. It is a non-tall P-ideal.
\item $\fin\otimes\{\emptyset\}$ is an ideal that is not tall and not a P-ideal.
\item $\{\emptyset\}\otimes\fin$ is a non-tall P-ideal.
\item $\fin\otimes\fin$ is a tall non-P-ideal.
\item $\I_{1/n}=\{A\subseteq\omega: \sum_{n\in A}\frac{1}{n+1}<+\infty\}$ is a tall P-ideal called \emph{the summable ideal}.
\item $\I_{d}=\{A\subseteq\omega: \lim_{n\to\infty}\frac{|A\cap n|}{n+1}=0\}$ is a tall P-ideal called the \emph{ideal of sets of asymptotic density zero}.
\item Let $\Omega$ be the set of all clopen subsets of the Cantor space $2^\omega$ having Lebesgue measure $1/2$ (note that $\Omega$ is countable).
Then the \emph{Solecki's ideal}, denoted by $\mathcal{S}$, is the collection of all subsets of $\Omega$ that can be covered by finitely many sets of the form $G_x=\{A\in\Omega: x\in A\}$ for $x\in 2^\omega$. $\mathcal{S}$ is a tall non-P-ideal.
\end{itemize}
\end{example}


\subsection{Ideal convergence}
\label{ConvDef}

Let $\I$ be an ideal on $\omega$. 
A sequence $(a_n)$ of reals is \emph{$\I$-convergent to zero} ($x_n\xrightarrow{\I}0$) if 
$$\{n\in\omega: |x_n|\geq \varepsilon\}\in\I \text{ for each $\varepsilon>0$}.$$
A sequence $(f_n) $ of real-valued functions  defined on $X$ is
\begin{itemize}
\item \emph{$\I$-pointwise convergent to zero} ($f_n\xrightarrow{\text{$\I$-p}}0$) if $f_n(x)\xrightarrow{\I}0$ for all $x\in X$ i.e.
$$\{n\in\omega: |f_n(x)|\geq\varepsilon\}\in\I \text{  for each $x\in X$ and  $\varepsilon>0$;}$$

\item \emph{$\I$-uniformly convergent to zero} ($f_n\xrightarrow{\text{$\I$-u}}0$) if 
$$\{n\in\omega: \exists x\in X\, (|f_n(x)|\geq\varepsilon)\}\in\I \text{  for each $\varepsilon>0$;}$$

\item \emph{$\I$-$\sigma$-uniformly convergent to zero} ($f_n\xrightarrow{\text{$\I$-$\sigma$-u}}0$) if there is a family $\{X_k:k\in \omega\}$ of subsets of $X$ such that 
$$\bigcup_{k\in\omega}X_k=X \text{ and } f_n\restriction X_k\xrightarrow{\text{$\I$-u}}0 \text{ for all $k\in\omega$;}$$
\item \emph{$\I$-quasi-normally convergent to zero} ($f_n\xrightarrow{\text{$\I$-qn}}0$) if there is a sequence  $(\varepsilon_n)$ of positive reals such that 
$$\varepsilon_n\xrightarrow{\I}0 \text{ and } \{n\in\omega: |f_n(x)-f(x)|\geq\varepsilon_n\}\in\I \text{ for every $x\in X$.}$$
\end{itemize}


\subsection{Spaces not distinguishing convergence}

For a topological space $X$, we write $\cC(X)$ to denote the family of all real-valued continuous functions defined on $X$.
Recall that a topological space $X$ is called a \emph{normal space} (or $T_4$-space) if $X$ is a Hausdorff space  and for every pair of disjoint closed subsets $A,B\subseteq  X$ there exist open sets $U,V$ such that $A\subseteq U$, $B \subseteq V$ and $U\cap V = \emptyset$.

\begin{definition}
Let $\alpha$ and $\beta$ be some notions of convergences of sequences of real-valued functions (for instance, pointwise, uniform, quasi-normal  or $\sigma$-uniform).
We write $f_n\xrightarrow{\alpha}0$ 
if  $(f_n)$ convergence to the constant zero function with respect to the notion $\alpha$.

\begin{enumerate}
    \item By $(\alpha,\beta)$ we denote the class of all  normal spaces  not distinguishing between $\alpha$ and $\beta$ convergences in $\cC(X)$ i.e.~a space  $X\in (\alpha,\beta)$ if and only if it is normal and 
    $$f_n\xrightarrow{\alpha}0 \iff f_n\xrightarrow{\beta}0
    \text{ for every sequence $(  f_n) $ in $\cC(X)$}.$$

\item By $\non(\alpha,\beta)$ we denote the smallest cardinality of a normal  space   which 
distinguishes between $\alpha$ and $\beta$ convergences in $\cC(X)$:
$$\non(\alpha,\beta) = \min\left(\{|X|:X\text{ is normal and } X\notin(\alpha,\beta)\}\cup\{\infty\}\right).$$
\end{enumerate}
\end{definition}

For instance, we write   $X \in (\text{$\I$-p,$\I$-u})$ if $X$ is normal and 
$$f_n\xrightarrow{\text{$\I$-p}}0 \iff f_n\xrightarrow{\text{$\I$-u}}0$$ 
for any sequence $(  f_n) $ of continuous real-valued functions defined on $X$.


\section{Spaces not distinguishing uniform convergence}
\label{sec:uniform-convergence}

\begin{proposition}
\label{prop:easy-implications-one-ideal}
Let $\I,\J$ be  ideals on $\omega$.
Let $X$ be a nonempty  topological space. 
Let  $(f_n)$ be a sequence in  $\cC(X)$.
\begin{enumerate}
    \item 
    $
f_n\xrightarrow{\text{$\I$-u}}0 
\implies 
f_n\xrightarrow{\text{$\I$-$\sigma$-u}}0
\implies 
f_n\xrightarrow{\text{$\I$-qn}}0
\implies 
f_n\xrightarrow{\text{$\I$-p}}0
.$\label{prop:easy-implications-one-ideal:item}

\item If $\I\subseteq\J$, then \label{prop:easy-implications-one-ideal:item-inclusion}

\begin{enumerate}
\item $f_n\xrightarrow{\text{$\I$-u}}0 \implies f_n\xrightarrow{\text{$\J$-u}}0$, 
\item $f_n\xrightarrow{\text{$\I$-$\sigma$-u}}0 \implies f_n\xrightarrow{\text{$\J$-$\sigma$-u}}0$, 
\item $f_n\xrightarrow{\text{$\I$-qn}}0 \implies f_n\xrightarrow{\text{$\J$-qn}}0$, 
\item $f_n\xrightarrow{\text{$\I$-p}}0 \implies f_n\xrightarrow{\text{$\J$-p}}0$.
\end{enumerate}

\end{enumerate}

\end{proposition}

\begin{proof}
(\ref{prop:easy-implications-one-ideal:item}) The first implication is obvious, the second is proved in \cite[Theorem~2.1 along with Note~2.1]{MR3038073}, whereas the third one is shown in \cite[Proposition 4.4]{MR3179991}.

(\ref{prop:easy-implications-one-ideal:item-inclusion}) 
Straightforward.
\end{proof}

\begin{proposition}
\label{prop:uniform-implies-pointwise}
\label{prop:uniform-implies-sigma-uniform}
\label{prop:sigma-uniform-implies-pointwise}
\label{prop:uniform-implies-quasinormal}
\label{prop:quasinormal-implies-pointwise}
\label{prop:sigma-uniform-implies-quasinormal}\label{prop:uniform-implies-pointwise:item}\label{prop:sigma-uniform-implies-pointwise:item}\label{prop:uniform-implies-quasinormal:item}\label{prop:quasinormal-implies-pointwise:item}\label{prop:sigma-uniform-implies-quasinormal:item}
    Let $\I$ and $\J$ be ideals on $\omega$.
Let $X$ be a nonempty topological space. The following conditions are equivalent.
\begin{enumerate}

    \item $\I\subseteq\J$.\label{prop:sigma-uniform-implies-quasinormal:ideals}\label{prop:uniform-implies-pointwise:ideals}\label{prop:uniform-implies-sigma-uniform:ideals}\label{prop:sigma-uniform-implies-pointwise:ideals}\label{prop:uniform-implies-quasinormal:ideals}\label{prop:quasinormal-implies-pointwise:ideals}

        \item $f_n\xrightarrow{\text{$\I$-u}}0 \implies f_n\xrightarrow{\text{$\J$-$\sigma$-u}}0$
    for every  sequence $(  f_n) $ in $\cC(X)$.\label{prop:uniform-implies-sigma-uniform:functions}

    \item $f_n\xrightarrow{\text{$\I$-u}}0 \implies f_n\xrightarrow{\text{$\J$-qn}}0$
    for every sequence $(  f_n) $ in $\cC(X)$.\label{prop:uniform-implies-quasinormal:functions}

    \item $f_n\xrightarrow{\text{$\I$-$\sigma$-u}}0 \implies f_n\xrightarrow{\text{$\J$-qn}}0$
    for every  sequence $(  f_n) $ in $\cC(X)$.\label{prop:sigma-uniform-implies-quasinormal:functions}

    \item $f_n\xrightarrow{\text{$\I$-$\sigma$-u}}0 \implies f_n\xrightarrow{\text{$\J$-p}}0$
    for every sequence $(  f_n) $ in $\cC(X)$.\label{prop:sigma-uniform-implies-pointwise:functions}

    \item $f_n\xrightarrow{\text{$\I$-u}}0 \implies f_n\xrightarrow{\text{$\J$-p}}0$
    for every sequence $(  f_n) $ in $\cC(X)$.\label{prop:uniform-implies-pointwise:functions}
    
    \item $f_n\xrightarrow{\text{$\I$-qn}}0 \implies f_n\xrightarrow{\text{$\J$-p}}0$
    for every  sequence $(  f_n) $ in $\cC(X)$.\label{prop:quasinormal-implies-pointwise:functions}

\end{enumerate}
The above characterizations are presented graphically on  Figure~\ref{fig:diag-1}.
\end{proposition}

\begin{figure}[h]
    \centering
\begin{tikzcd}[row sep=7em,column sep=7em]
\text{$\I$-p}  \arrow[r,leftarrow, "\I\supseteq\cL"] & 
\text{$\cL$-u} \arrow[d,"\K\supseteq \cL"]   \arrow[ld,sloped,near start,"\J\supseteq\cL"]\\
\text{$\J$-qn} \arrow[u, "\I\supseteq\J"] \arrow[r, leftarrow,"\J\supseteq \K"]& 
\text{$\K$-$\sigma$-u}   \arrow[lu,sloped, near start, "\I\supseteq \K"]
\end{tikzcd}
\caption{Diagram for Proposition~\ref{prop:uniform-implies-pointwise}, where 
``$\text{$\I$-p}\xleftarrow{\I\supseteq\cL} \text{$\cL$-u}$'' is a counterpart of the equivalence ``$(\ref{prop:uniform-implies-pointwise:ideals})\iff (\ref{prop:uniform-implies-pointwise:functions})$'', and similarly for other arrows.}
    \label{fig:diag-1}
\end{figure}
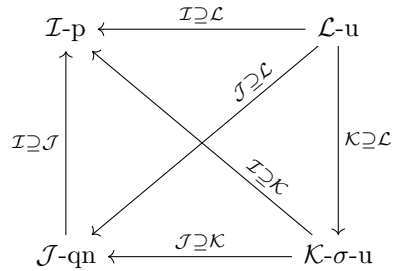

\begin{proof}
First, we see that it is enough to prove the following chains of implications:
\begin{itemize}
    \item 
    $
    (\ref{prop:sigma-uniform-implies-quasinormal:ideals})
    \implies 
    (\ref{prop:uniform-implies-sigma-uniform:functions}) 
    \implies 
    (\ref{prop:uniform-implies-quasinormal:functions})
    \implies 
    (\ref{prop:uniform-implies-pointwise:functions})
    \implies 
    (\ref{prop:sigma-uniform-implies-quasinormal:ideals})
    $,

    \item 
    $
    (\ref{prop:sigma-uniform-implies-quasinormal:ideals})
    \implies 
    (\ref{prop:sigma-uniform-implies-quasinormal:functions}) 
    \implies 
(\ref{prop:sigma-uniform-implies-pointwise:functions})
\implies 
    (\ref{prop:uniform-implies-pointwise:functions})
\implies 
    (\ref{prop:sigma-uniform-implies-quasinormal:ideals})
    $,

    \item 
    $
    (\ref{prop:sigma-uniform-implies-quasinormal:ideals})
    \implies 
    (\ref{prop:quasinormal-implies-pointwise:functions})
    \implies 
    (\ref{prop:sigma-uniform-implies-quasinormal:ideals})
$.
\end{itemize}

Second, we observe  that the following implications easily follow from Proposition~\ref{prop:easy-implications-one-ideal}:
\begin{itemize}
    \item 
$(\ref{prop:sigma-uniform-implies-quasinormal:ideals})
\implies 
(\ref{prop:uniform-implies-sigma-uniform:functions})$, 
$(\ref{prop:uniform-implies-sigma-uniform:functions}) 
\implies 
(\ref{prop:uniform-implies-quasinormal:functions})$, 
$(\ref{prop:uniform-implies-quasinormal:functions})
\implies 
(\ref{prop:uniform-implies-pointwise:functions})$,

\item 
$(\ref{prop:sigma-uniform-implies-quasinormal:ideals})
\implies 
(\ref{prop:sigma-uniform-implies-quasinormal:functions})$, 
$(\ref{prop:sigma-uniform-implies-quasinormal:functions}) 
\implies 
(\ref{prop:sigma-uniform-implies-pointwise:functions})$,
$(\ref{prop:sigma-uniform-implies-pointwise:functions})
\implies 
(\ref{prop:uniform-implies-pointwise:functions})$, 

\item 
$(\ref{prop:sigma-uniform-implies-quasinormal:ideals})
\implies 
(\ref{prop:quasinormal-implies-pointwise:functions})$.
\end{itemize}

Third,  we prove the remaining two implications:
$(\ref{prop:uniform-implies-pointwise:functions})
\implies 
(\ref{prop:sigma-uniform-implies-quasinormal:ideals})$
and 
$(\ref{prop:quasinormal-implies-pointwise:functions})
\implies 
(\ref{prop:sigma-uniform-implies-quasinormal:ideals})$ simultaneously. 
Let $A\in \I$.
We define $f_n:X\to\R$ by $f_n(x)=\chf_{A}(n)$ for every $x\in X$. Then $f_n$ are constant so continuous.
Since 
$f_n\xrightarrow{\text{$\I$-u}}0$ and $f_n\xrightarrow{\text{$\I$-qn}}0$, both (\ref{prop:uniform-implies-pointwise:functions}) and 
(\ref{prop:quasinormal-implies-pointwise:functions})
imply that 
 $f_n\xrightarrow{\text{$\J$-p}}0$.
Take any  $x_0\in X$. Then $A = \{n\in \omega: |f_n(x_0)|>1/2\}\in \J$.
\end{proof}

\begin{proposition}
\label{prop:sigma-uniform-implies-uniform}
\label{prop:pointwise-implies-uniform}
\label{prop:quasinormal-implies-uniform}
    Let $\I$ and $\J$ be ideals on $\omega$.
Let $X$ be a nonempty normal space.
 The following conditions are equivalent.
\begin{enumerate}
    \item $|X|<\omega$ and $\I\subseteq \J$.\label{prop:sigma-uniform-implies-uniform:ideals}\label{prop:pointwise-implies-uniform:ideals}\label{prop:quasinormal-implies-uniform:ideals}

    \item $f_n\xrightarrow{\text{$\I$-p}}0 \implies f_n\xrightarrow{\text{$\J$-u}}0$
    for every sequence $(  f_n) $ in $\cC(X)$.\label{prop:pointwise-implies-uniform:functions}

    \item $f_n\xrightarrow{\text{$\I$-qn}}0 \implies f_n\xrightarrow{\text{$\J$-u}}0$
    for every  sequence $(  f_n) $ in $\cC(X)$.\label{prop:quasinormal-implies-uniform:functions}

    \item $f_n\xrightarrow{\text{$\I$-$\sigma$-u}}0 \implies f_n\xrightarrow{\text{$\J$-u}}0$
    for every sequence $(  f_n) $ in $\cC(X)$.\label{prop:sigma-uniform-implies-uniform:functions}
\end{enumerate}

The above characterizations are presented graphically on  Figure~\ref{fig:diag-2}.
\end{proposition}

\begin{figure}[h]
\centering
\begin{tikzcd}[row sep=7em,column sep=7em]
\text{$\I$-p}  \arrow[r, "|X|<\omega \  \land \  \I\subseteq\cL"] & 
\text{$\cL$-u}  \\
\text{$\J$-qn}  \arrow[ru,sloped,"|X|<\omega\ \land\ \J\subseteq\cL"]& 
\text{$\K$-$\sigma$-u} \arrow[u,swap,"|X|<\omega \ \land \ \K\subseteq \cL"] 
\end{tikzcd}
\caption{Diagram for Proposition~\ref{prop:pointwise-implies-uniform}, where 
``$\text{$\I$-p}\xrightarrow{|X|<\omega  \land \I\subseteq\cL} \text{$\cL$-u}$'' is a counterpart of  the equivalence ``$(\ref{prop:pointwise-implies-uniform:ideals})\iff (\ref{prop:pointwise-implies-uniform:functions})$'',  and similarly for other arrows.}
    \label{fig:diag-2}
\end{figure}

\begin{proof}
    $(\ref{prop:pointwise-implies-uniform:ideals})\implies (\ref{prop:pointwise-implies-uniform:functions})$
    Let $(  f_n) $ be a sequence in $\cC(X)$ such that $f_n\xrightarrow{\text{$\I$-p}}0$.
    Let $\varepsilon>0$.
    For every $x\in X$, $A_x=\{n\in \omega:|f_n(x)|>\varepsilon\}\in \I$.
    Since $X$ is finite and $\I\subseteq\J$, $A=\bigcup\{A_x:x\in X\}\in \J$.
    But $\{n\in \omega: \exists x\in X\, (|f_n(x)|>\varepsilon)\} =A$, so  $f_n\xrightarrow{\text{$\J$-u}}0$.

    $(\ref{prop:pointwise-implies-uniform:functions})\implies(\ref{prop:quasinormal-implies-uniform:functions})$
It easily follows from Proposition~\ref{prop:easy-implications-one-ideal}.

$(\ref{prop:quasinormal-implies-uniform:functions})\implies (\ref{prop:sigma-uniform-implies-uniform:functions})$
It easily follows from Proposition~\ref{prop:easy-implications-one-ideal}.

$(\ref{prop:sigma-uniform-implies-uniform:functions})\implies (\ref{prop:sigma-uniform-implies-uniform:ideals})$
First, we show that $\I\subseteq\J$.
Let $A\in \I$.
We define $f_n:X\to\R$ by $f_n(x)=\chf_{A}(n)$ for every $x\in X$. Then $f_n$ are constant so continuous and  
$f_n\xrightarrow{\text{$\I$-$\sigma$-u}}0$.
Thus
$f_n\xrightarrow{\text{$\J$-u}}0$.
 Then $A = \{n\in \omega: \exists x\in X\, (|f_n(x)|>1/2)\}\in \J$.

Second, we show that $X$ is finite. Suppose, for sake of contradiction, that $X$ is infinite. 
Since $X$ is an infinite Hausdorff space, it is not difficult 
to show that there is an infinite sequence $(U_n:n\in \omega) $ of pairwise disjoint nonempty open subsets of $X$ (see e.g.~\cite[Theorem~12.1, p.~45]{MR0776620}).   
For each $n\in \omega$, we pick $x_n\in U_n$. 
Since $X$ is a normal space, we can use Urysohn's Lemma to obtain that for every $n$  there is a continuous function $f_n:X\to[0,1]$ such that $f_n(x_n)=1$ and $f_n(x)=0$ for every $x\in X\setminus U_n$.
If we show that 
$f_n\xrightarrow{\text{$\I$-$\sigma$-u}}0$ holds 
but 
$f_n\xrightarrow{\text{$\J$-u}}0$ does not hold, we obtain a contradiction and the proof will be finished.

Let us show $f_n\xrightarrow{\text{$\I$-$\sigma$-u}}0$.
We put $X_0=X\setminus \bigcup\{U_k:k<\omega\}$ and $X_{k+1}=U_k$ for every $k\in \omega$.
Then  $X$ is covered by $\{X_k:k\in \omega\}$.
Since  $f_n\restriction X_{0}$ is a constant  function with value zero for every $n$,  $f_n\restriction X_{0}\xrightarrow{\text{$\I$-u}}0$.
Whereas for $k\in \omega$, $f_n\restriction X_{k+1}$ is a constant function with value zero for every $n\neq k$, so  $f_n\restriction X_{k+1}\xrightarrow{\text{$\I$-u}}0$.

To show  that $f_n\xrightarrow{\text{$\J$-u}}0$ does not hold, it is enough to see that  $\{n\in \omega: \exists x\in X\, (|f_n(x)|>1/2)\} 
\supseteq 
\{n\in \omega: f_n(x_n)=1\}
=
\omega\notin \J$.
\end{proof}

\begin{corollary}
\label{cor:pointwise-versus-uniform-two-ideals}
    Let $\I$ and $\J$ be ideals on $\omega$.
Let $X$ be a nonempty normal space.
The following conditions are equivalent.
\begin{enumerate}
    \item $|X|<\omega$ and $\I= \J$.
    \item $f_n\xrightarrow{\text{$\I$-p}}0 \iff f_n\xrightarrow{\text{$\J$-u}}0$
    for every  sequence $(  f_n) $ in $\cC(X)$.
    \item $ f_n\xrightarrow{\text{$\I$-qn}}0 \iff f_n\xrightarrow{\text{$\J$-u}}0$
    for every  sequence $(  f_n) $ in $\cC(X)$.
    \item $f_n\xrightarrow{\text{$\I$-$\sigma$-u}}0 \iff f_n\xrightarrow{\text{$\J$-u}}0$
    for every sequence $(  f_n) $ in $\cC(X)$.
\end{enumerate}
\end{corollary}

\begin{proof}
It follows from Propositions~\ref{prop:uniform-implies-pointwise} and \ref{prop:sigma-uniform-implies-uniform}.
\end{proof}  

\begin{corollary}
\label{cor:quasinormal-versus-uniform}
\label{cor:pointwise-versus-uniform}
    Let $\I$ and $\J$ be ideals on $\omega$.
Let $X$ be a normal space.
\begin{enumerate}
\item 
If $\I\neq \J$, then $\non(\text{$\I$-p,$\J$-u})=\non(\text{$\I$-qn,$\J$-u})=\non(\text{$\I$-$\sigma$-u,$\J$-u})=1$.
\item $X\in (\text{$\I$-p,$\I$-u})
    \iff 
    X\in (\text{$\I$-qn,$\I$-u}) 
    \iff 
    X\in (\text{$\I$-$\sigma$-u,$\I$-u}) 
    \iff
    |X|<\omega$.\label{cor:only-one-class:uniform}

\item $\non(\text{$\I$-p,$\I$-u})=\non(\text{$\I$-qn,$\I$-u})=\non(\text{$\I$-$\sigma$-u,$\I$-u})=\omega$.\label{cor:pointwise-versus-uniform:non}
\item There is no infinite normal space  in the classes $(\text{$\I$-p,$\I$-u})$,  $(\text{$\I$-qn,$\I$-u})$,  $(\text{$\I$-$\sigma$-u,$\I$-u})$.\label{cor:pointwise-versus-uniform:space}

\end{enumerate}
\end{corollary}

\begin{proof}
    It follows from Corollary~\ref{cor:pointwise-versus-uniform-two-ideals}.
\end{proof}

\begin{corollary}
\label{cor:only-one-class}
Let $\I$ be an ideal on $\omega$.
Let $X$ be a normal space.
\begin{enumerate}
 
\item 
$X\in (\text{$\I$-p,$\I$-$\sigma$-u})$ $\iff$ $X\in (\text{$\I$-p,$\I$-qn})$ and $X\in (\text{$\I$-qn,$\I$-$\sigma$-u})$.\label{cor:only-one-class:nonuniform}

\item 
$\non(\text{$\I$-p,$\I$-$\sigma$-u}) = \min\{\non(\text{$\I$-p,$\I$-qn}) , \non(\text{$\I$-qn,$\I$-$\sigma$-u})\}$.\label{cor:only-one-class:nonuniform-non}

\end{enumerate}

\end{corollary}

\begin{proof}

(\ref{cor:only-one-class:nonuniform})
Since the implication ``$\impliedby$'' is obvious, we only show the reversed one. Assume that 
$X\in (\text{$\I$-p,$\I$-$\sigma$-u})$.

First we will show that $X$ is in the class ($\I$-p,$\I$-qn). By Proposition~\ref{prop:easy-implications-one-ideal}, if $f_n\xrightarrow{\text{$\I$-qn}}0$ then $f_n\xrightarrow{\text{$\I$-p}}0$, for every sequence $(  f_n) $ in $\cC(X)$. On the other hand, if $(  f_n) \in\cC(X)$ is such that $f_n\xrightarrow{\text{$\I$-p}}0$ then $f_n\xrightarrow{\text{$\I$-$\sigma$-u}}0$ (as $X$ is in the class ($\I$-p,$\I$-$\sigma$-u)), so also $f_n\xrightarrow{\text{$\I$-qn}}0$ (by Proposition~\ref{prop:easy-implications-one-ideal}).

Now we show that $X$ is in the class ($\I$-qn,$\I$-$\sigma$-u). By Proposition~\ref{prop:easy-implications-one-ideal}, if $f_n\xrightarrow{\text{$\I$-$\sigma$-u}}0$ then $f_n\xrightarrow{\text{$\I$-qn}}0$, for every sequence $(  f_n) $ in $\cC(X)$. On the other hand, if $(  f_n) \in\cC(X)$ is such that $f_n\xrightarrow{\text{$\I$-qn}}0$ then $f_n\xrightarrow{\text{$\I$-p}}0$ (by Proposition~\ref{prop:easy-implications-one-ideal}), so also $f_n\xrightarrow{\text{$\I$-$\sigma$-u}}0$  (as $X$ is in the class ($\I$-p,$\I$-$\sigma$-u)).

(\ref{cor:only-one-class:nonuniform-non}) 
It follows from item (\ref{cor:only-one-class:nonuniform}).
\end{proof}


\section{Spaces not distinguishing  \texorpdfstring{$\sigma$}{sigma}-uniform convergence}
\label{sec:sigma-uniform}

In the sequel, we use the convention that $\min \emptyset = \infty$ and $\kappa<\infty$ for every cardinal $\kappa$.

\begin{notation}
	Let $\I$ be an ideal on $\omega$.
\begin{enumerate}
    \item 
	$ \widehat{\cP}_\I = \{(  A_n)  \in \I^\omega : A_n\cap A_k=\emptyset \text{ for all distinct $n, k$}\}$.

 \item $\cP_\I =\{ (  A_n )   \in \widehat{\cP}_\I: \bigcup\{A_n: n\in \omega\} = \omega\}$.

 \item $\cM_\I= \left\{(  E_k) \in\I^\omega: \forall k\in\omega\, (E_k\subseteq E_{k+1})\right\}.$
 
\end{enumerate}
\end{notation}

\begin{definition}
Let $\I, \J, \K$ be ideals  on $\omega$.
\begin{enumerate}
    \item 
$\bnumber_s(\I,\J,\K) = 
			\min \{|\cE|:\cE\subseteq\widehat{\cP}_\K \land \forall (  A_n) \in\cP_\J \, \exists (  E_n) \in\cE \, (\bigcup_{n\in\omega}(A_{n+1}\cap \bigcup_{i\leq n}E_i)\notin\I)\}.$

\item  $\bs(\I,\J) = \min \{|\cE|: \cE\subseteq\cM_\I \land \forall (  A_n) \in\cM_\J \, \exists (  E_n) \in\cE \, \exists^\infty n\, (E_n\not\subseteq A_n)\}.$

\item $\add_\omega(\I,\J) = \min\{|\cA|:\cA\subseteq\I\land \forall  (  B_n) \in \J^{\omega}\, \exists A\in \cA\, \forall n\in \omega\,(A\not\subseteq B_n))\}.$

\end{enumerate}

  In the sequel, we will use the following shorthands: 
  $\bnumber_s(\I)=\bnumber_s(\I,\I,\I)$, 
  $\bs(\I)=\bs(\I,\I)$, 
  $\add_\omega(\I)=\add_\omega(\I,\I)$.
\end{definition}

The cardinal $\bnumber_s(\I,\J,\K)$ was introduced by Staniszewski~\cite[p.~1184]{MR3624786}  to characterize the smallest size of a space which is not $(\I,\J,\K)$-QN. Later Repick\'{y}~\cite{MR4336563,MR4336564}, among others, characterized the same class of spaces in terms of another cardinal. In \cite{MR3423409}, \v{S}upina introduced the cardinal $\kappa(\I,\J)$ which is equal to $\bnumber_s(\J,\J,\I)$.
In the case of maximal ideal, $\bb_s(\I,\I,\I)$ and $\bb_s(\I,\fin,\fin)$ were studied by Canjar~\cite{CanjarPhD,MR0924678,MR1036675}. In the case of Borel ideals, $\bb_s(\I,\I,\I)$ and $\bb_s(\I,\fin,\fin)$ were extensively studied in \cite{MR4472525}.

The cardinals $\bs(\I,\J)$ and $\add_\omega(\I,\J)$ are introduced here but the latter cardinal appeared, in a sense, in \cite{MR3624786} were the author introduced the notion of $\kappa$-P$(\Fin,\J)$-ideals, because  it is not difficult to see that
 $\add_\omega(\I,\J) = \min\{\kappa: \I \text{ is not } \kappa \text{-P}(\Fin,\J)\}$.

\begin{theorem}
\label{thm:pointwise-implies-quasinormal:necessary-condition}
\label{thm:quasinormal-implies-sigma-uniform:necessary-condition}
\label{thm:pointwise-implies-sigma-uniform:necessary-condition}
\label{thm:pointwise-implies-quasinormal}
\label{thm:pointwise-implies-sigma-uniform}
\label{thm:quasinormal-implies-sigma-uniform}

    Let $\I,\J,\K$ be ideals on $\omega$.
Let $X$ be a nonempty topological space.

\begin{enumerate}
\item In the following list of conditions, each implies the next.\label{thm:pointwise-implies-quasinormal:necessary-condition:item}

\begin{enumerate}
    
\item  $|X|<\bb_s(\J,\J,\I)$.\label{thm:pointwise-implies-quasinormal:necessary-condition:item-cardinal}
\item  
 $f_n\xrightarrow{\text{$\I$-p}}0 \implies f_n\xrightarrow{\text{$\J$-qn}}0$
  for every  sequence $(  f_n) $ in $\cC(X)$.\label{thm:pointwise-implies-quasinormal:necessary-condition:item-functions}

\item $\I\subseteq\J$.\label{thm:pointwise-implies-quasinormal:necessary-condition:item-ideals}
\end{enumerate}

\item In the following list of conditions, each implies the next.\label{thm:quasinormal-implies-sigma-uniform:necessary-condition:item}

\begin{enumerate}
\item  $|X|<\add_\omega(\J,\K)$.\label{thm:quasinormal-implies-sigma-uniform:necessary-condition:item-cardinal}

\item 
 $f_n\xrightarrow{\text{$\J$-qn}}0 \implies f_n\xrightarrow{\text{$\K$-$\sigma$-u}}0$
  for every  sequence $(  f_n) $ in $\cC(X)$.\label{thm:quasinormal-implies-sigma-uniform:necessary-condition:item-functions}

\item $\J\subseteq\K$.\label{thm:quasinormal-implies-sigma-uniform:necessary-condition:item-ideals}
\end{enumerate}

\item In the following list of conditions, each implies the next.\label{thm:pointwise-implies-sigma-uniform:necessary-condition:item}

\begin{enumerate}

\item  $|X|<\bs(\I,\K)$.\label{thm:pointwise-implies-sigma-uniform:necessary-condition:item-cardinal} 

\item 
 $f_n\xrightarrow{\text{$\I$-p}}0 \implies f_n\xrightarrow{\text{$\K$-$\sigma$-u}}0$
   for every  sequence $(  f_n) $ in $\cC(X)$.\label{thm:pointwise-implies-sigma-uniform:necessary-condition:item-functions}

\item $\I\subseteq\K$.\label{thm:pointwise-implies-sigma-uniform:necessary-condition:item-ideals}
\end{enumerate}
\end{enumerate}
The above  implications are presented graphically on  Figure~\ref{fig:diag-3}.
\end{theorem}

\begin{figure}[h]
    \centering
\begin{tikzcd}[row sep=7em,column sep=7em]
\text{$\I$-p}  \arrow[d,swap, dashed, "|X|<\bnumber_s(\J{,}\J{,}\I)"] \arrow[dr, dashed, "|X|<\bs(\I{,}\K)"]  & 
\text{$\cL$-u} \\
\text{$\J$-qn} \arrow[r, dashed, "|X|<\add_\omega(\J{,}\K)"] & 
\text{$\K$-$\sigma$-u}
\end{tikzcd}
\caption{Diagram for Theorem~\ref{thm:pointwise-implies-quasinormal:necessary-condition}, where 
``$\text{$\J$-qn}\xdashrightarrow{|X|<\add_\omega(\J,\K) } \text{$\K$-$\sigma$-u}$'' is a counterpart of the implication ``$(\ref{thm:quasinormal-implies-sigma-uniform:necessary-condition:item-cardinal})\implies(\ref{thm:quasinormal-implies-sigma-uniform:necessary-condition:item-functions})$'', and similarly for other arrows.}
    \label{fig:diag-3}
\end{figure}
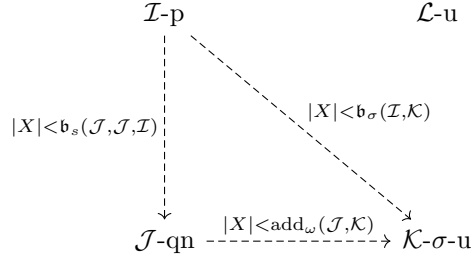

\begin{proof}
$(\ref{thm:pointwise-implies-quasinormal:necessary-condition:item-cardinal})
\implies (\ref{thm:pointwise-implies-quasinormal:necessary-condition:item-functions})$
It follows from \cite[Theorems~5.1 and 6.2]{MR3423409}.

$(\ref{thm:pointwise-implies-quasinormal:necessary-condition:item-functions})
\implies (\ref{thm:pointwise-implies-quasinormal:necessary-condition:item-ideals})$
Let $A\in \I$.
We define $f_n:X\to\R$ by $f_n(x)=\chf_{A}(n)$ for every $x\in X$. Then $f_n$ are constant so continuous and  
$f_n\xrightarrow{\text{$\I$-p}}0$.
Thus
$f_n\xrightarrow{\text{$\J$-qn}}0$.
Then there exists a sequence $(\varepsilon_n)$ of positive reals which is $\J$-convergent to zero and 
$\{n\in \omega:|f_n(x)|\geq \varepsilon_n\}\in \J$ for every $x\in X$.
Let $x_0\in X$.
Then 
$
A 
=
\{n\in \omega:|f_n(x_0)|>1/2\} 
\subseteq 
\{n\in \omega:|f_n(x_0)|>\varepsilon_n\land \varepsilon_n<1/2\}
\cup
\{n\in \omega:\varepsilon_n\geq 1/2\}
\subseteq 
\{n\in \omega:|f_n(x_0)|>\varepsilon_n\}
\cup
\{n\in \omega:\varepsilon_n\geq 1/2\}
\in \J $.

$(\ref{thm:quasinormal-implies-sigma-uniform:necessary-condition:item-cardinal})\implies(\ref{thm:quasinormal-implies-sigma-uniform:necessary-condition:item-functions})$
If $\J\not\subseteq\K$, then it is easy to see that $\add_\omega(\J,\K)=1$. (Indeed, let $E\in \I\setminus \J$ and $\cE=\{E\}$. Take any $(A_n)\in \cM_\J$. Then $E\not\subseteq A_n$ for every $n\in \omega$.) Hence, there is nothing to prove in that case. Below we assume that $\J\subseteq\K$.

Suppose that $|X|<\add_\omega(\J,\K)$ and let $(  f_n) $ be a sequence in $\cC(X)$ such that $f_n\xrightarrow{\text{$\J$-qn}}0 $.
Then there exists a sequence $(\varepsilon_n)$ of positive reals which is $\J$-converegnt to zero and 
$\{n\in \omega: |f_n(x)|\geq \varepsilon_n\}\in \J$ for every $x\in X$.
We define
$E^x=\left\{n\in\omega: |f_n(x)|\geq\varepsilon_n\right\}$ for every $x\in X$.
Since $\{E^x:x\in X\}\subseteq\J$ and  $|X|<\add_\omega(\J,\K)$, there is $\cB =\{B_k:k\in\omega\}\subseteq \K$ such that for each $x\in X$ there is $k\in\omega$ with  $E^x\subseteq B_k$. 
We define  $X_k=\{x\in X: E^x\subseteq B_k\}$ for each $k\in \omega$.
It is easy to see that   $X = \bigcup\{X_k:k\in\omega\}$, and we show that  $f_n\restriction X_k$ converges $\K$-uniformly to $0$  for every $k\in\omega$. 
Fix any $k\in\omega$ and $\varepsilon>0$. 
Since $\J\subseteq\K$ and $\varepsilon_n\xrightarrow{\J}0$, the set $C_\varepsilon = \{n\in \omega:\varepsilon_n\geq \varepsilon\}\in \K$.
For every $x\in X_k$, we have 
$
\{n\in \omega:|f_n(x)|\geq \varepsilon\} 
\subseteq 
\{n\in \omega:|f_n(x)|\geq \varepsilon_n\land \varepsilon>\varepsilon_n\} 
\cup
\{n\in \omega:\varepsilon_n\geq \varepsilon\}
\subseteq 
E^x\cup C_\varepsilon 
\subseteq B_k\cup C_\varepsilon$.
Consequently, 
$\{n\in \omega:\exists x\in X_k\,(|f_n(x)|\geq \varepsilon)\} 
\subseteq B_k\cup C_\varepsilon\in \K$.

$(\ref{thm:quasinormal-implies-sigma-uniform:necessary-condition:item-functions})\implies(\ref{thm:quasinormal-implies-sigma-uniform:necessary-condition:item-ideals})$
Let $A\in \J$.
We define $f_n:X\to\R$ by $f_n(x)=\chf_{A}(n)$ for every $x\in X$. Then $f_n$ are constant so continuous and  
$f_n\xrightarrow{\text{$\J$-qn}}0$.
Thus
$f_n\xrightarrow{\text{$\K$-$\sigma$-u}}0$.
Then there exists a cover  $\{X_k:k\in \omega\}$ of $X$ such that  
$f_n\restriction X_k\xrightarrow{\text{$\K$-u}}0$ for every $k\in \omega$.
Let $x_0\in X$ and $k_0\in \omega$ be such that $x_0\in X_{k_0}$.
Then 
$
A 
=
\{n\in \omega:|f_n(x_0)|>1/2\} 
\subseteq 
\{n\in \omega:\exists x\in X_{k_0}\,(|f_n(x)|>1/2)\}
\in \K$.

$
(\ref{thm:pointwise-implies-sigma-uniform:necessary-condition:item-cardinal})
\implies
(\ref{thm:pointwise-implies-sigma-uniform:necessary-condition:item-functions})
$
Suppose that $|X|<\bs(\I,\K)$ and let $(  f_n) $ be a sequence in $\cC(X)$ such that $f_n\xrightarrow{\text{$\I$-p}}0 $.
For every $x\in X$ and $k\in\omega$ define:
$$E_k^x=\left\{n\in\omega: |f_n(x)|\geq\frac{1}{k+1}\right\}.$$
Observe that $E^x_{k}\in\I$ and $E_k^x\subseteq E^x_{k+1}$ for all $x\in X$ and $k\in\omega$, i.e.,  $(  E^x_k) \in\cM_\I$ for all $x\in X$.

Since the family $\cE=\left\{(E^x_k) : x\in X\right\}$ has cardinality $|X|<\bs(\I,\K)$, there is $(  A_k) \in\cM_\K$ such that for each $x\in X$ there is $m_x\in\omega$ such that $E^x_k\subseteq A_k$ for all $k\geq m_x$. Define $X_m=\{x\in X: m=m_x\}$ and note that $\bigcup_{m\in\omega}X_m=X$. 

We claim that $f_n\restriction X_m$ converges $\K$-$\sigma$-uniformly to $0$  for every $m\in\omega$. Fix any $m\in\omega$ and $\varepsilon>0$. Let $k\in\omega$ be such that $k\geq m$ and $\frac{1}{k+1}<\varepsilon$. Since $A_k\in\K$, to finish the proof it suffices to show that $|f_n(x)|<\varepsilon$ for every $x\in X_m$ and $n\in\omega\setminus A_k$. Fix $x\in X_m$ and $n\in\omega\setminus A_k$. Since $k\geq m=m_x$, we have $E^x_k\subseteq A_k$. Hence, $\omega\setminus E^x_k\supseteq \omega\setminus A_k\ni n$. Thus, $|f_n(x)|<\frac{1}{k+1}<\varepsilon$ and we are done.

$
(\ref{thm:pointwise-implies-sigma-uniform:necessary-condition:item-functions})
\implies
(\ref{thm:pointwise-implies-sigma-uniform:necessary-condition:item-ideals})
$
It follows from item (\ref{thm:pointwise-implies-quasinormal:necessary-condition:item}), because
$f_n\xrightarrow{\text{$\K$-$\sigma$-u}}0
\implies 
f_n\xrightarrow{\text{$\K$-qn}}0$ by Proposition~\ref{prop:easy-implications-one-ideal}.
\end{proof}

\begin{corollary}
    Let $\I$ and $\J$  be ideals on $\omega$.
If $\I\neq\J$, then 
 $\non(\text{$\I$-p,$\J$-$\sigma$-u})=\non(\text{$\I$-p,$\J$-qn})=\non(\text{$\I$-qn,$\J$-$\sigma$-u})=1$.
\end{corollary}

\begin{proof}
It follows from Proposition~\ref{prop:uniform-implies-pointwise} and Theorem~\ref{thm:pointwise-implies-quasinormal:necessary-condition}.
\end{proof}

\begin{proposition}
\label{prop:quasinormal-convergence:restriction-to-less-than-bs-sets}
\label{prop:sigma-uniform-convergence:restriction-to-less-than-bsigma-sets}
    Let $\I$ be an ideal on $\omega$.
    Let $X$ be a topological space and suppose that $X=\bigcup\{X_\alpha:\alpha<\kappa\}$.
Let $(f_n)$ be a sequence in $\cC(X)$.
    \begin{enumerate}
        \item If $\kappa<\bnumber_s(\I)$ and $f_n\restriction X_\alpha \xrightarrow{\text{$\I$-qn}}0$ for every $\alpha<\kappa$, then 
        $f_n\xrightarrow{\text{$\I$-qn}}0$.\label{prop:quasinormal-convergence:restriction-to-less-than-bs-sets:item}

        \item If $\kappa<\bs(\I)$ and $f_n\restriction X_\alpha \xrightarrow{\text{$\I$-$\sigma$-u}}0$ for every $\alpha<\kappa$, then 
        $f_n\xrightarrow{\text{$\I$-$\sigma$-u}}0$.\label{prop:sigma-uniform-convergence:restriction-to-less-than-bsigma-sets:item}
\end{enumerate}
\end{proposition}

\begin{proof}
(\ref{prop:quasinormal-convergence:restriction-to-less-than-bs-sets:item})
For each $\alpha<\kappa$, there is a sequence $(\varepsilon^\alpha_n)$ of positive reals which is $\I$-convergent to zero and
$A_{x,\alpha} = \{n\in \omega: |f_n(x)|\geq \varepsilon^\alpha_n\}\in \I$
for every $x\in X_\alpha$.
For each $n\in \omega$, we define  $\phi_n:\kappa\to \R$ by
$\phi_n(\alpha)=\varepsilon^\alpha_n$ for each $\alpha\in \kappa$.
Having the discrete topology on $\kappa$, functions $\phi_n$ are continuous.
Since $\phi_n\xrightarrow{\text{$\I$-p}}0$ and $\kappa<\bnumber_s(\I)$, we obtain that
$\phi_n\xrightarrow{\text{$\I$-qn}}0$ (by Proposition~\ref{thm:pointwise-implies-quasinormal:necessary-condition}(\ref{thm:pointwise-implies-quasinormal:necessary-condition:item})).
Thus, there is a sequence $(\varepsilon_n)$ of positive reals which is $\I$-convergent to zero and $B_\alpha = \{n\in \omega: |\phi_n(\alpha)|\geq \varepsilon_n\}\in \I$ for every $\alpha\in \kappa$.
We claim that the sequence $(\varepsilon_n)$ also witnesses 
$f_n\xrightarrow{\text{$\I$-qn}}0$.
Take any $x\in X$.
There is $\alpha<\kappa$ with $x\in X_\alpha$.
Then 
$
\{n\in \omega: |f_n(x)|\geq \varepsilon_n\}
\subseteq 
\{n:|f_n(x)|\geq \varepsilon^\alpha_n\land \varepsilon^\alpha_n<\varepsilon_n\}
\cup
\{n\in \omega: \varepsilon^\alpha_n\geq \varepsilon_n\}
\subseteq A_{x,\alpha}\cup B_\alpha\in \I
$.

(\ref{prop:sigma-uniform-convergence:restriction-to-less-than-bsigma-sets:item})
If $\kappa$ is finite, then the result is obvious. If $\kappa$ is infinite, then $\kappa\cdot \omega=\kappa$, so without loss of generality we can assume that 
$f_n\restriction X_\alpha \xrightarrow{\text{$\I$-u}}0$ for every $\alpha<\kappa$.
Now,  we define $A^\alpha_k=\{n\in \omega:\exists x\in X_\alpha\,(|f_n(x)|>\frac{1}{k+1})\}$ for every $\alpha<\kappa$ and $k\in \omega$.
Since $(A^\alpha_k)\in \cM_\I$ for every $\alpha<\kappa$ and $\kappa<\bs(\I)$, there is $(B_n)\in \cM_\I$ such that 
for each $\alpha<\kappa$ there is $k_\alpha\in \omega$ such that 
$A^\alpha_k\subseteq B_k$ for every $k\geq k_\alpha$.
For each $k\in \omega$, we define
$Y_k = \bigcup\{X_\alpha: k_\alpha=k\}$.
Then $X=\bigcup\{Y_k:k\in\omega\}$, and once we show that  $f_n\restriction Y_k  \xrightarrow{\text{$\I$-u}}0$ for each $k\in \omega$, the proof will be finished.
Take any $k\in \omega$ and $\varepsilon>0$. Let $i\in \omega$ be such that $\varepsilon>\frac{1}{i+1}$ and $i\geq k$.
Then 
$
\{n\in \omega: \exists x\in Y_k\,(|f_n(x)|\geq \varepsilon)\}
\subseteq 
\{n\in \omega: \exists x\in Y_k\,(|f_n(x)|\geq \frac{1}{i+1})\}
\subseteq 
\{n\in \omega: \exists \alpha<\kappa\, \exists x\in X_\alpha\,(k_\alpha=k\land |f_n(x)|\geq \frac{1}{i+1})\}
\subseteq B_i\in \I
$.
\end{proof}

\begin{theorem}
\label{thm:pointwise-implies-quasinormal:sufficient-condition-for-discrete-space}
\label{thm:pointwise-implies-sigma-uniform:sufficient-condition-for-discrete-space}
\label{thm:quasinormal-implies-sigma-uniform:sufficient-condition-for-discrete-space}
Let $\I,\J,\K$ be ideals on $\omega$.
Let $X$ be a \underline{discrete} topological space.
\begin{enumerate}

\item The following conditions are equivalent.\label{thm:pointwise-implies-quasinormal:sufficient-condition-for-discrete-space:item}
\begin{enumerate}
\item 
$f_n\xrightarrow{\text{$\I$-p}}0 \implies f_n\xrightarrow{\text{$\J$-qn}}0$
for any sequence $(  f_n) $ in $\cC(X)$.\label{thm:pointwise-implies-quasinormal:sufficient-condition-for-discrete-space:item-functions}
\item 
$|X|<\bb_s(\J,\J,\I)$.\label{thm:pointwise-implies-quasinormal:sufficient-condition-for-discrete-space:item-cardinal}
\end{enumerate}

\item The following conditions are equivalent.\label{thm:pointwise-implies-sigma-uniform:sufficient-condition-for-discrete-space:item}
\begin{enumerate}
\item 
$f_n\xrightarrow{\text{$\I$-p}}0 \implies f_n\xrightarrow{\text{$\K$-$\sigma$-u}}0$
for any sequence $(  f_n) $ in $\cC(X)$.\label{thm:pointwise-implies-sigma-uniform:sufficient-condition-for-discrete-space:item-functions}
\item   $|X|<\bs(\I,\K)$.\label{thm:pointwise-implies-sigma-uniform:sufficient-condition-for-discrete-space:item-cardinal}
\end{enumerate}

\item The following conditions are equivalent.\label{thm:quasinormal-implies-sigma-uniform:sufficient-condition-for-discrete-space:item}
\begin{enumerate}
\item 
$f_n\xrightarrow{\text{$\J$-qn}}0 \implies f_n\xrightarrow{\text{$\K$-$\sigma$-u}}0$
  for any sequence $(  f_n) $ in $\cC(X)$.\label{thm:quasinormal-implies-sigma-uniform:sufficient-condition-for-discrete-space:item-functions}
  \item 
 $|X|<\add_\omega(\J,\K)$.\label{thm:quasinormal-implies-sigma-uniform:sufficient-condition-for-discrete-space:item-cardinal}
\end{enumerate}

\end{enumerate}
\end{theorem}

\begin{proof}
(\ref{thm:pointwise-implies-quasinormal:sufficient-condition-for-discrete-space:item})
It follows from \cite[Theorems~5.1 and 6.2]{MR3423409} and \cite[Theorem~4.9(1)]{MR3624786} as the property $W(\J,\J,\I)$ from \cite{MR3624786} is equivalent to $\J$ being a "weak P($\I$)-ideal" from \cite{MR3423409}.

$(\ref{thm:pointwise-implies-sigma-uniform:sufficient-condition-for-discrete-space:item-functions}) \implies (\ref{thm:pointwise-implies-sigma-uniform:sufficient-condition-for-discrete-space:item-cardinal}) $ 
 Enumerate $X=\{x_\alpha: \alpha<|X|\}$ and fix any $\cE=\{(  E^\alpha_k) : \alpha<|X|\}\subseteq\cM_\I$. We need to show that $\cE$ is not a witness for $\bs(\I,\K)$, i.e.~there is  $(  A_k) \in\cM_\K$ such that for each $\alpha<|X|$ there is $m\in\omega$ such that $E^\alpha_k\subseteq A_k$ for all $k\geq m$. 

Define functions $f_n:X\to\mathbb{R}$ by:
$$f_n(x_\alpha)=
\begin{cases}
\frac{1}{k+1}, & \text{if }n\in E^\alpha_k\setminus E^\alpha_{k-1},\\
0, & \text{otherwise}\\
\end{cases}$$
 for every $\alpha<|X|$ (here we put $E^\alpha_{-1}=\emptyset$). 
Since $X$ is \emph{discrete}, functions $f_n$ are continuous for every $n$.
Observe that $f_n\xrightarrow{\text{$\I$-p}}0$, since for each $x\in X$ and $k\in\omega$ we have:
$$\left\{n\in\omega: |f_n(x)|\geq\frac{1}{k+1}\right\}=E^\alpha_k\in\I,$$
where $\alpha<|X|$ is given by $x=x_\alpha$.

By our assumption, $f_n\xrightarrow{\text{$\K$-$\sigma$-u}}0$. Thus, there is $(  X_m) \subseteq\cP(X)$ such that $\bigcup_m X_m=X$ and $f_n\restriction X_m\xrightarrow{\text{$\K$-u}}0$ for all $m\in\omega$, i.e.,
$$B_{m,k}=\left\{n\in\omega: \exists x\in X_m\, \left(|f_n(x)|\geq\frac{1}{k+1}\right)\right\}\in\K$$
for every $k,m\in\omega$.

Define $A_k=B_{0,k}\cup B_{1,k}\cup\ldots\cup B_{k,k}\in\K$ for all $k\in\omega$. Note that $A_k\subseteq B_{0,k+1}\cup B_{1,k+1}\cup\ldots\cup B_{k,k+1}\subseteq A_{k+1}$ for every $k\in\omega$. We claim that $(  A_k) \in\cM_\K$ is as needed, i.e., for each $\alpha<|X|$ there is $m\in\omega$ such that $E^\alpha_k\subseteq A_k$ for all $k\geq m$. 

Fix $\alpha<|X|$ and let $m\in\omega$ be such that $x_\alpha\in X_m$. Fix any $k\geq m$ and $n\in E^\alpha_k$. Then $f_n(x_\alpha)\geq\frac{1}{k+1}$. Since $x_\alpha\in X_m$ and $k\geq m$, $n\in B_{m,k}\subseteq B_{0,k}\cup B_{1,k}\cup\ldots\cup B_{k,k}=A_k$. As $n$ was arbitrary, we can conclude that $E^\alpha_k\subseteq A_k$. This finishes the proof.

$(\ref{thm:pointwise-implies-sigma-uniform:sufficient-condition-for-discrete-space:item-cardinal}) \implies (\ref{thm:pointwise-implies-sigma-uniform:sufficient-condition-for-discrete-space:item-functions})$
It follows from Theorem~\ref{thm:pointwise-implies-sigma-uniform:necessary-condition}(\ref{thm:pointwise-implies-sigma-uniform:necessary-condition:item}).

$(\ref{thm:quasinormal-implies-sigma-uniform:sufficient-condition-for-discrete-space:item-functions})
\implies 
(\ref{thm:quasinormal-implies-sigma-uniform:sufficient-condition-for-discrete-space:item-cardinal})
$
 Enumerate $X=\{x_\alpha: \alpha<|X|\}$ and fix any $\cA=\{ A_\alpha: \alpha<|X|\}\subseteq \J$. 
 We need to show that $\cA$ is not a witness for $\add_\omega(\J,\K)$, i.e.~there is  $\{B_k:k\in\omega\}\subseteq \K$ such that for each $\alpha<|X|$ there is $k\in\omega$ such that $A_\alpha\subseteq B_k$. 

We define functions $f_n:X\to\R$ by
$$f_n(x_\alpha)=
\chf_{A_\alpha}(n)$$
 for every $\alpha<|X|$. 
Since $X$ is \emph{discrete}, functions $f_n$ are continuous for every $n$.
Observe that $f_n\xrightarrow{\text{$\J$-qn}}0$.
Indeed, if we take any  sequence $(\varepsilon_n)$ of positive reals which is ordinary convergent to zero, then 
for each $x\in X$ there is $\alpha$ with $x=x_\alpha$ and 
$
\left\{n\in\omega: |f_n(x_\alpha)|\geq\varepsilon_n\right\}
=
\left\{n\in A_\alpha : |f_n(x_\alpha)|\geq\varepsilon_n\right\}
\cup
\left\{n\in\omega\setminus A_\alpha: |f_n(x_\alpha)|\geq\varepsilon_n\right\}
=
\left\{n\in A_\alpha : 1\geq\varepsilon_n\right\}
\cup
\left\{n\in\omega\setminus A_\alpha: 0\geq\varepsilon_n\right\}
\subseteq A_\alpha\cup \emptyset\in \J
$.

By our assumption, $f_n\xrightarrow{\text{$\K$-$\sigma$-u}}0$. Thus, there is a covering $\{X_k :k\in \omega\}$ of $X$  such that  $f_n\restriction X_k\xrightarrow{\text{$\K$-u}}0$ for all $k\in\omega$.

For each $k\in \omega$, we define 
$$B_k = \left\{n\in \omega: \exists x\in X_k\, \left(|f_n(x)|>\frac{1}{2}\right)\right\}.$$
We see that $B_k\in \K$ for each $k\in\omega$, and we claim that for every $A\in \cA$ there is $k$ with $A\subseteq B_k$.
Indeed, let $A\in \cA$. Let $\alpha$ be such that $A=A_\alpha$.
Then there is $k\in \omega$ such that $x_\alpha \in X_k$.
Let $n\in A_\alpha$. Then $f_n(x_\alpha)=1>1/2$, so $n\in B_k$.

$(\ref{thm:quasinormal-implies-sigma-uniform:sufficient-condition-for-discrete-space:item-cardinal})
\implies 
(\ref{thm:quasinormal-implies-sigma-uniform:sufficient-condition-for-discrete-space:item-functions})
$
It follows from Theorem~\ref{thm:quasinormal-implies-sigma-uniform:necessary-condition}(\ref{thm:quasinormal-implies-sigma-uniform:necessary-condition:item}).
\end{proof}

In \cite{MR1129696}, the authors proved that $\non(\text{$\Fin$-p,$\Fin$-qn})=\bnumber$ i.e.~the smallest size of non-QN-spaces equals $\bnumber$.
The following corollary is a counterpart of the above result which gives a purely combinatorial characterization of the topological cardinal characteristics $\non(\text{$\I$-p,$\I$-qn})$, $\non(\text{$\I$-p,$\I$-$\sigma$-u})$, $\non(\text{$\I$-qn,$\I$-$\sigma$-u})$
with the aid of other bounding-like numbers.

\begin{corollary}
\label{cor:quasinormal-versus-sigma-uniform:non}
\label{cor:pointwise-versus-quasinormal:non}
\label{cor:pointwise-versus-sigma-uniform:non}
    Let $\I$ be an ideal on $\omega$.
\begin{enumerate}

\item 
 $\non(\text{$\I$-p,$\I$-$\sigma$-u})=\bs(\I)$.\label{cor:pointwise-versus-sigma-uniform:non-equals-bsigma}

\item 
 $\non(\text{$\I$-p,$\I$-qn})=\bb_s(\I)$.\label{cor:pointwise-versus-quasinormal:non-equals-bs}

\item 
 $\non(\text{$\I$-qn,$\I$-$\sigma$-u})=\add_\omega(\I)$.\label{cor:quasinormal-versus-sigma-uniform:non-equals-add-omega}

\end{enumerate}
\end{corollary}

\begin{proof}
(\ref{cor:pointwise-versus-sigma-uniform:non-equals-bsigma})
The inequality $\non(\text{$\I$-p,$\I$-$\sigma$-u})\geq \bs(\I)$
 follows from Proposition~\ref{prop:sigma-uniform-implies-pointwise}
and Theorem~\ref{thm:pointwise-implies-sigma-uniform:necessary-condition}.
On the other hand, 
if $X$ is a discrete topological space of cardinality $\bs(\I)$, then 
by Theorem~\ref{thm:pointwise-implies-sigma-uniform:sufficient-condition-for-discrete-space}, $X$ is not in $(\text{$\I$-p,$\I$-$\sigma$-u})$.  
Consequently,  $\non(\text{$\I$-p,$\I$-$\sigma$-u})\leq \bs(\I)$.

Items (\ref{cor:pointwise-versus-quasinormal:non-equals-bs}) and (\ref{cor:quasinormal-versus-sigma-uniform:non-equals-add-omega}) can be proved in the same way.
\end{proof}

In Section~\ref{sec:spaces-of-arbitrary-cardinality},  we show that
 we \emph{cannot} add an item: ``there is no space  of cardinality $\bs(\I)$ in $(\text{$\I$-p,$\I$-$\sigma$-u})$'' in Corollary~\ref{cor:pointwise-versus-sigma-uniform:non} (in contrast with Corollary~\ref{cor:pointwise-versus-uniform}).


\section{Properties of cardinals describing minimal size of spaces  distinguishing convergence}
\label{sec:properties-of-bsigma}

In this section we will take a closer look on the cardinals $\bb_s(\I,\J,\K)$,  $\bs(\I,\J)$ and $\add_\omega(\I,\J)$.

The following easy proposition shows that these  cardinals are coordinate-wise monotone (increasing or decreasing depending on a coordinate).

\begin{proposition}
\label{prop:monotonicity-of-b-numbers}
	Let $\I,\I',\J,\J',\K,\K'$ be ideals on $\omega$.
	\begin{enumerate}
		\item If $\I\subseteq\I'$, then 
		$\bb_s(\I,\J,\K)\leq \bb_s(\I',\J,\K)$, 
  $\bs(\I,\J)\geq\bs(\I',\J)$ and $\add_\omega(\I,\J)\geq \add_\omega(\I',\J)$.
		\item If $\J\subseteq\J'$, then 
		$\bb_s(\I,\J,\K)\leq \bb_s(\I,\J',\K)$, 
  $\bs(\I,\J)\leq\bs(\I,\J')$ and $\add_\omega(\I,\J)\leq \add_\omega(\I,\J')$.
		\item If $\K\subseteq\K'$, then 
		$\bb_s(\I,\J,\K)\geq \bb_s(\I,\J,\K')$.
\end{enumerate}
\end{proposition}

The following theorem reveals the relationship between the considered cardinals. 

\begin{theorem}
\label{thm:bsigma-leq-min-bs-and-add-omega}
\label{thm:bsigma-leq-add-omega}
\label{thm:bsigma-leq-bs}
Let $\I,\J$ be ideals on $\omega$.
\begin{enumerate}
    \item 
$\bs(\I,\J) = \min\{ \bb_s(\I\cap \J,\J,\I), \add_\omega(\I,\J)\}$.\label{thm:bsigma-leq-min-bs-and-add-omega:item}
   
\item 
$\bs(\I) = \min\{\bb_s(\I),\add_\omega(\I)\}$.\label{thm:bsigma-leq-min-bs-and-add-omega:item-for-one-ideal}

\end{enumerate}
\end{theorem}

\begin{proof}
(\ref{thm:bsigma-leq-min-bs-and-add-omega:item}, $\leq$)
First, we show $\bs(\I,\J)\leq \bb_s(\I\cap \J,\J,\I)$. 
Let $\cE = \{(  E^\alpha_n:n\in \omega) :\alpha<\bs(\I,\J)\}$ be a ``witness'' for $\bnumber_s(\I\cap \J,\J,\I)$ i.e.~$(  E^\alpha_n:n\in \omega) \in \widehat{\cP}_\I$ for every $\alpha$ and
for every $(  A_n)  \in \cP_\J$ there is $\alpha$ with $\bigcup_{n\in\omega}(A_{n+1}\cap \bigcup_{i\leq n}E^\alpha_i)\notin\I\cap \J$.
For every $\alpha<\bnumber_s(\K,\J,\I)$ and $n\in \omega$, we define
$F^\alpha_n = \bigcup_{i\leq n}E^\alpha_i$.
Then $(  F^\alpha_n:n\in \omega) \in \cM_\I$, and we claim that $\{(  F^\alpha_n) :\alpha<\bnumber_s(\I\cap \J,\J,\I)\}$ is a ``witness'' for $\bs(\I,\J)$ i.e.~for every $(  A_n) \in \cM_\J$ there is $\alpha$ such that  $F^\alpha_n\not\subseteq A_n$
for infinitely many $n$.
Indeed, take any  $(  A_n) \in \cM_\J$. Without loss of generality, we can assume that $n\in A_n$ for every $n\in \omega$.
We define $B_0=A_0$ and $B_n=A_n\setminus A_{n-1}$ for  $n\geq 1$.
Then $(  B_n) \in \cP_\J$, so there is $\alpha$ with  $\bigcup_{n\in\omega}(B_{n+1}\cap \bigcup_{i\leq n}E^\alpha_i)\notin\I\cap \J$.
Now, suppose for sake of contradiction that 
$F^\alpha_n\subseteq A_n$ for almost all $n\in \omega$, say for all $n>n_0$.
Then $B_{n+1}\cap F^\alpha_n=\emptyset$ for every $n>n_0$.
Consequently, $B_{n+1}\cap \bigcup_{i\leq n}E^\alpha_i =\emptyset$ for every $n>n_0$. Thus, 
$\bigcup_{n\in\omega}(B_{n+1}\cap \bigcup_{i\leq n}E^\alpha_i) 
\subseteq 
\bigcup_{n\leq n_0}(B_{n+1}\cap \bigcup_{i\leq n}E^\alpha_i)
\in \I\cap \J$, a contradiction.

Second, we show $ \bs(\I,\J)\leq \add_\omega(\I,\J)$.
Let $\cA=\{A_\alpha:\alpha<  \add_\omega(\I,\J)\}$ be a ``witness'' for $\add_\omega(\I,\J)$ i.e.~$A_\alpha\in \I$ for every $\alpha$ and for every $\{B_n:n\in \omega\}\subseteq \J$ there is $\alpha$ such that $A_\alpha\not\subseteq B_n$ for every $n\in \omega$.
For every $\alpha<\add_\omega(\I,\J)$  and $n\in \omega$, we define  $E^\alpha_n=A_\alpha$.
Then $(  E^\alpha_n:n\in\omega) \in \cM_\I$, and we claim that 
$\{(  E^\alpha_n:n\in\omega) :\alpha<\add_\omega(\I,\J)\}$ is a ``witness'' for $\bs(\I,\J)$ i.e.~for every $(  B_n) \in \cM_\J$ there is $\alpha$ such that $E^\alpha_n\not\subseteq B_n$ for infinitely many $n$.
Indeed, take any $(  B_n) \in \cM_\J$ then $\{B_n:n\in \omega\}\subseteq\J$, so there is $\alpha$
such that $A_\alpha\not\subseteq B_n$ for every $n\in \omega$.
Since $E^\alpha_n=A_\alpha$ for every $n$, we obtain $E^\alpha_n\not\subseteq B_n$ for every $n\in \omega$.

(\ref{thm:bsigma-leq-min-bs-and-add-omega:item}, $\geq$)
Let $\kappa < \min\{ \bb_s(\I\cap \J,\J,\I), \add_\omega(\I,\J)\}$.
If we show that $\kappa <\bs(\I,\J)$, the proof will be finished.
We take any $\cE=\{(  E^\alpha_n:n\in\omega) :\alpha<\kappa\}\subseteq\cM_\I$ and need to find $(  A_n) \in \cM_\J$ such that for every $\alpha<\kappa$ we have $E^\alpha_n\subseteq A_n$ for all but finitely many $n\in \omega$.
For every $\alpha<\kappa$ and $n\in \omega$, we define $F^\alpha_n=E^\alpha_n\setminus \bigcup_{i<n}E^\alpha_i$.
Since $(  F^\alpha_n:n\in\omega)  \in \widehat{\cP}_\I$ for every $\alpha<\kappa$ and $\kappa<\bnumber_s(\I\cap \J,\J,\I)$, we obtain $(  B_n:n\in \omega)  \in \cP_\J$ such that  $G_\alpha = \bigcup_{n<\omega}(B_{n+1}\cap E^\alpha_n) =\bigcup_{n<\omega}(B_{n+1}\cap \bigcup_{i\leq n}F^\alpha_i)\in \I\cap\J$
for every $\alpha$.
Since $G_\alpha\in \I$ for every $\alpha<\kappa$ and $\kappa< \add_\omega(\I,\J)$, we obtain $(  C_n:n\in \omega)  \in \J^\omega$ such that for every $\alpha<\kappa$ there is $n_\alpha\in \omega$ with $G_\alpha\subseteq C_{n_\alpha}$. 
For every $n\in \omega$, we define $A_n = \bigcup_{i\leq n}(B_i\cup C_i)$.
Then $(  A_n:n\in \omega)  \in \cM_\J$ and we claim that 
for every $\alpha<\kappa$
 we have $E^\alpha_n\subseteq A_n$   for all but finitely many $n\in \omega$.
 Indeed, take any $\alpha<\kappa$ and notice that 
$$E^\alpha_n 
\subseteq 
\bigcup_{i\leq n} B_i \cup \bigcup_{k\geq n}(B_{k+1}\cap E^\alpha_k)
\subseteq 
\bigcup_{i\leq n} B_i \cup G_\alpha
\subseteq 
\bigcup_{i\leq n} B_i \cup \bigcup_{i\leq n}C_i = A_n
$$
for  every $n\geq n_\alpha$.

(\ref{thm:bsigma-leq-min-bs-and-add-omega:item-for-one-ideal})
It follows from item (\ref{thm:bsigma-leq-min-bs-and-add-omega:item}), but one could also show it ``topologically'' by using Corollaries~\ref{cor:only-one-class}(\ref{cor:only-one-class:nonuniform-non}) and \ref{cor:pointwise-versus-sigma-uniform:non}.
\end{proof}

The following proposition reveals some bounds for the considered cardinals.
In this proposition we use some known cardinals considered in the literature so far which we define first.

For any ideal $\I$, we define
$$\adds(\I)= \min\{|\cA|:\cA\subseteq \I\land \forall B\in  \I\, \exists A\in \cA\,(|A\setminus B|=\omega)\}.$$

For $f,g\in \omega^\omega$ we write $f\leq^*g$ if $f(n)\leq g(n)$ for all but finitely many $n\in \omega$.
The \emph{bounding number} $\bnumber$ is the smallest size of $\leq^*$-unbounded subset of $\omega^\omega$:
$$\bnumber = \min\{|\cF|: \cF\subseteq\omega^\omega  \land  \neg(\exists g\in \omega^\omega\, \forall f\in \cF\, (f\leq^* g))\}.$$

\begin{proposition}
\label{prop:bounds-for-bs}
\label{prop:bounds-for-add-omega}
\label{prop:bounds-for-bsigma}
\label{prop:bounds-for-bsigma-bs}
\label{prop:bsigma-is-regular}
\label{prop:bsigma-has-uncountable-cofinality}
 Let $\I,\J,\K$ be ideals on $\omega$.
\begin{enumerate}

\item \label{prop:bounds-for-bs:item}
\begin{enumerate}
\item If $\I\not\subseteq\J$, then $\bb_s(\I\cap \J,\J,\I)=1$.
\item If $\I\subseteq\J$, then $\bb_s(\I\cap \J,\J,\I)\geq \omega_1$.\label{prop:bounds-for-bs:item:bs-uncountable}
\item $\omega_1\leq\bb_s(\I)\leq \continuum$.\label{prop:bounds-for-bs:item:bs-is-leq-continuum}

\item $\bnumber_s(\Fin,\J,\Fin)=\bnumber$\label{prop:bounds-for-bs:item:equals-b}.
\end{enumerate}

\item \label{prop:bounds-for-add-omega:item}
\begin{enumerate}
\item If $\I\not\subseteq \J$, then $\add_\omega(\I,\J)=1$.\label{prop:bounds-for-add-omega:item:notincluded-equal-one}
\item If $\I\subseteq \J$, then $\add_\omega(\I,\J)\geq\max \{\omega_1,\adds(\I)\}$.\label{prop:bounds-for-add-omega:item:included-geq-omega-one}
\label{prop:bounds-for-add-omega:item:add-omega-geq-add-star}

\item
$\add_\omega(\I)<\infty \iff \I \text{ is not countably generated}$.\label{prop:bounds-for-add-omega:item:add-omega-finite}

\end{enumerate}
     
\item \label{prop:bounds-for-bsigma:item}
\begin{enumerate}
\item $\bs(\fin,\J)=\bb$.\label{prop:bounds-for-bsigma:item:FIN}
\item If $\I\not\subseteq\J$ then $\bs(\I,\J)=1$.\label{prop:bounds-for-bsigma:item:one}

\item If $\I\subseteq\J$ then $\omega_1\leq \bs(\I,\J)\leq\bb$.\label{prop:bounds-for-bsigma:item:leq-b}\label{prop:bsigma-has-uncountable-cofinality:item}\label{prop:bounds-for-bsigma:item:geq-omega-one} 
\item If $\I\subseteq\J$ then $\cf(\bs(\I,\J))\geq\omega_1$.\label{prop:bounds-for-bsigma:item:bsigma-has-uncountable-cofinality}

\end{enumerate}

\item $\bs(\I)\geq\bb_s(\fin,\I,\I)=\min \{\bb,\adds(\I)\}$.\label{prop:bounds-for-bs:item:equals-min-b-and-adds}\label{prop:bounds-for-bsigma-bs:item}

\end{enumerate}
\end{proposition}

\begin{proof}
(\ref{prop:bounds-for-bs:item}) 
See \cite[Proposition~3.13 and Theorem~4.2]{MR4472525}.

(\ref{prop:bounds-for-add-omega:item:notincluded-equal-one})
Let $E\in \I\setminus \J$. 
Let $\cE=\{E\}$ and take any $(A_n)\in \cM_\J$. Then $E\not\subseteq A_n$ for every $n\in \omega$ (otherwise, $E\subseteq A_n\in \J$ would imply $E\in \J$). Thus, $\add_\omega(\I,\J)\leq 1$.

(\ref{prop:bounds-for-add-omega:item:included-geq-omega-one})
The inequality $\add_\omega(\I,\J)\geq\omega_1$ will follow from item (\ref{prop:bounds-for-bsigma:item:leq-b}) and Theorem~\ref{thm:bsigma-leq-bs}. To show that $\add_\omega(\I,\J)\geq\adds(\I)$, let $\cA\subseteq\I$ be a witness for $\add_\omega(\I,\J)$. We claim that $\cA$ is also a witness for $\adds(\I)$.
Indeed, take any $B\in \I$. Let $\Fin=\{F_n:n\in \omega\}$ and define $B_n=B\cup F_n$ for every $n\in \omega$.
Since $\I\subseteq\J$, we have  $(B_n)\in [\J]^\omega$. Consequently,  there is $A\in \cA$ such that $A\not\subseteq B_n=B\cup F_n$ for any $n\in \omega$.
Thus, $|A\setminus B|=\omega$.

(\ref{prop:bounds-for-add-omega:item:add-omega-finite}) 
Straightforward.

(\ref{prop:bounds-for-bsigma:item:FIN})
The inequality  $\bs(\fin,\J)\leq \bnumber$ follows from item (\ref{prop:bounds-for-bs:item:equals-b}) and Theorem~\ref{thm:bsigma-leq-bs}. Below we show $\bnumber\leq \bs(\fin,\J)$. 
Using Proposition~\ref{prop:monotonicity-of-b-numbers}, we see that it is enough to show $\bnumber\leq \bs(\fin)$.
Fix any $\cE=\{(  E^\alpha_k) : \alpha<\bs(\fin)\}\subseteq\cM_\fin$ which is a witness for $\bs(\fin)$. For each $\alpha<\bs(\fin)$, we  define a function $f_\alpha\in\omega^\omega$ by $f_\alpha(k)=\max E^\alpha_k$. 
We claim that $\{f_\alpha:  \alpha<\bs(\fin)\}$ is $\leq^*$-unbounded subset of $\omega^\omega$. Fix any $g\in\omega^\omega$. We want to find $ \alpha<\bs(\fin)$ such that $f_\alpha\not\leq^* g$. Without loss of generality we may assume that $g$ is increasing. Define $A_k=\{i\in\omega: i\leq g(k)\}$ for all $k\in\omega$. Then $(  A_k) \in\cM_\fin$. Since $\cE$ is a witness for $\bs(\fin)$, there is $ \alpha<\bs(\fin)$ such that $E^\alpha_k\not\subseteq A_k$ for infinitely many $k\in\omega$. Observe that $E^\alpha_k\not\subseteq A_k$ implies $g(k)<f_\alpha(k)$. Hence, $g(k)<f_\alpha(k)$ for infinitely many $k\in\omega$, which means that  $f_\alpha\not\leq^\star g$.

(\ref{prop:bounds-for-bsigma:item:one}) 
It follows from item (\ref{prop:bounds-for-add-omega:item:notincluded-equal-one}) and Theorem~\ref{thm:bsigma-leq-bs}.

(\ref{prop:bounds-for-bsigma:item:geq-omega-one})
The inequality $\bs(\I,\J)\leq\bb$ follows from item (\ref{prop:bounds-for-bsigma:item:FIN}) and Proposition~\ref{prop:monotonicity-of-b-numbers}. Below we show $\bs(\I,\J)\geq\omega_1$. 

Fix any $\{(  E^n_k) : k\in\omega\}\subseteq\cM_\I$. We will find $(  A_k) \in\cM_\J$ such that $\{k\in\omega: E^n_k\not\subseteq A_k\}\in\fin$ for all $n\in\omega$. 

Define $A_k=E^0_k\cup E^1_k\cup\ldots\cup E^k_k$ for all $k\in\omega$. Then $A_k\in\I\subseteq\J$ and $A_k\subseteq E^0_{k+1}\cup E^1_{k+1}\cup\ldots\cup E^k_{k+1}\subseteq A_{k+1}$ as $(  E^n_k) \in\cM_\I$ for each $n\in\omega$. Moreover, for each $n\in\omega$ and $k\geq n$ we have $E^n_k\subseteq A_k$. Hence, $(  A_k) \in\cM_\J$ is as needed.

(\ref{prop:bounds-for-bsigma:item:bsigma-has-uncountable-cofinality})
Let $\cE$ be a witness for $\bs(\I,\J)$ i.e.~$|\cE|=\bs(\I,\J)$, $\cE\subseteq\cM_\I$ and for every $(A_n)\in \cM_\J$ there is $(E_n)\in \cE$ such that $E_n\not\subseteq A_n$ for infinitely many $n\in \omega$.
Now, suppose for sake of contradiction that $\cf(\bs(\I,\J))=\omega$.
Using the properties of cofinality, we know that $\cE$ can be decomposed  into the union of  countably many subfamilies $\cE_k$ of cardinalites less than $\bs(\I,\J)$.
Since $|\cE_k|<\bs(\I,\J)$, there is $(A_n^k)\in \cM_\J$ such that for every $(E_n)\in \cE_k$ we have $E_n\subseteq A_n^k$ for all but finitely many $n\in \omega$.
Then $\cA=\{(A_n^k):k\in \omega\}\subseteq\cM_\J$ and $|\cA|\leq \omega<\bs(\J)$ (by item (\ref{prop:bounds-for-bsigma:item:geq-omega-one})), so there is $(B_n)\in \cM_\J$ such that for every  $k \in\omega$ we have  $A_n^k \subseteq B_n$ for all but finitely many $n\in \omega$.
Consequently, for every $(E_n)\in \cE$ we have $E_n\subseteq B_n$ for all but finitely many $n\in \omega$, a contradiction with the choice of the family $\cE$.

(\ref{prop:bounds-for-bsigma-bs:item})
The equality $\bb_s(\fin,\I,\I)=\min\{\bb,\adds(\I)\}$ is shown in \cite[Theorem~4.8]{MR4472525}. Below we show that $\bs(\I)\geq\bb_s(\fin,\I,\I)$.

Let $\cE=\{\{ E^\alpha_n:n\in \omega\}: \alpha<\bs(\I)\}\subseteq\cM_\I$ be a witness for $\bs(\I)$. 
We define $F^\alpha_0=E^\alpha_0$ and 
$F_n^\alpha = E^\alpha_n\setminus E^\alpha_{n-1}$
for every $\alpha<\bs(\I)$ and $n\geq 1$.
Then $\cF=\{F^\alpha_n:n\in \omega\}:\alpha<\bs(\I)\}\subseteq \widehat{\cP}_\I$, and we claim that $\cF$ is a witness for $\bnumber_s(\Fin,\I,\I)$.
Indeed, take any $(A_n)\in \cP_\I$. For every $n\in \omega$, we define $B_n=\bigcup_{i\leq n}A_i$. Then $(B_n)\in \cM_\I$, so there exists $\alpha$ such that $E^\alpha_n\not\subseteq B_n$ for infinitely many $n$.
Let $(k_n)$ be a strictly increasing sequence such that 
$E^\alpha_{k_n}\not\subseteq B_{k_n}$ for every $n\in \omega$.
Thus, for every $n\in \omega$ there is $l_n>k_n$ and $a_n\in A_{l_n}\cap E^\alpha_{k_n}$.
Then $A=\{a_n:n\in\omega\}$ is infinite.
If we show  that $A\subseteq \bigcup_{n<\omega}(A_{n+1}\cap \bigcup_{i\leq n}F^\alpha_n)$, the proof will be finished.
Take any $a_n\in A$. Then
$a_n\in A_{l_n}\cap E^\alpha_{k_n} = A_{l_n}\cap \bigcup_{i\leq k_n}F^\alpha_{i} \subseteq A_{l_n}\cap \bigcup_{i< l_n}F^\alpha_{i}
=
A_{(l_n-1)+1}\cap \bigcup_{i\leq l_n-1}F^\alpha_{i}$. 
\end{proof}

\begin{corollary}
    \label{cor:bounds-for-bsigma}
For every ideal $\I$ on $\omega$ we have
$$\omega_1\leq\bs(\I)=\min\{\bb_s(\I),\add_\omega(\I)\}\leq\bb.$$
\end{corollary}

\begin{proof}
It follows from Theorem~\ref{thm:bsigma-leq-bs}  and Proposition~\ref{prop:bounds-for-bsigma}(\ref{prop:bounds-for-bsigma:item:leq-b}).
\end{proof}

\begin{corollary}
\label{cor:bs-is-regular}
The cardinals  $\bnumber_s(\I)$, $\bs(\I)$ and $\add_\omega(\I)$ are regular for every ideal $\I$.
\end{corollary}

\begin{proof}
The regularity of $\bnumber_s(\I)$ is shown in \cite[Corollary~3.12]{MR4472525} (however, one could also show it using a similar ``topological'' argument as for $\bs(\I)$ presented below). 

We will present two proofs of regularity of $\bs(\I)$ -- one ``topological'' and one ``purely combinatorial''. We start with the ``topological'' proof.

Suppose for sake of contradiction that $\bs(\I)=\bigcup\{A_\alpha:\alpha<\kappa\}$ where $\kappa<\bs(\I)$ and $|A_\alpha|<\bs(\I)$ for every $\alpha<\kappa$. 
Let $X$ be a normal space such  that $X\notin (\text{$\I$-p,$\I$-$\sigma$-u})$ and
$|X|=\bs(\I)$ (which exists by Corollary 
\ref{cor:pointwise-versus-sigma-uniform:non}(\ref{cor:pointwise-versus-sigma-uniform:non-equals-bsigma})).
Then we can write $X =\bigcup\{X_\alpha:\alpha<\kappa\}$ with $|X_\alpha|=|A_\alpha|$ for each $\alpha<\kappa$.
Take a sequence $(f_n)$ in $\cC(X)$ such that $f_n\xrightarrow{\text{$\I$-p}}0$
but $f_n\xrightarrow{\text{$\I$-$\sigma$-u}}0$ does not hold.
Since  
$f_n\restriction X_\alpha \xrightarrow{\text{$\I$-p}}0$ and $|X_\alpha|<\bs(\I)$ for every $\alpha<\kappa$, we can use  Theorem~\ref{thm:pointwise-implies-sigma-uniform:necessary-condition}(\ref{thm:pointwise-implies-sigma-uniform:necessary-condition:item}) to obtain that 
$f_n\restriction X_\alpha \xrightarrow{\text{$\I$-$\sigma$-u}}0$ for every $\alpha<\kappa$.
Now, Proposition~\ref{prop:sigma-uniform-convergence:restriction-to-less-than-bsigma-sets}(\ref{prop:sigma-uniform-convergence:restriction-to-less-than-bsigma-sets:item}) 
implies that 
$f_n\xrightarrow{\text{$\I$-$\sigma$-u}}0$, a contradiction.  

Now we present the ``purely combinatorial'' proof of regularity of $\bs(\I)$. Let $\cE$ be a witness for $\bs(\I)$ i.e.~$|\cE|=\bs(\I)$, $\cE\subseteq\cM_\I$ and for every $(A_n)\in \cM_\I$ there is $(E_n)\in \cE$ such that $E_n\not\subseteq A_n$ for infinitely many $n\in \omega$.
Using the properties of cofinality, we know that $\cE$ can be decomposed  into the union of  $\cf(\bs(\I))$ subfamilies $\cE_\alpha$ of cardinalites less than $\bs(\I)$.
Since $|\cE_\alpha|<\bs(\I)$, there is $(A_n^\alpha)\in \cM_\I$ such that for every $(E_n)\in \cE_\alpha$ we have $E_n\subseteq A_n^\alpha$ for all but finitely many $n\in \omega$.
Now, suppose for sake of contradiction that $\bs(\I)$ is not regular 
i.e.~$\cf(\bs(\I))<\bs(\I)$.
Then $\cA=\{(A_n^\alpha):\alpha<\cf(\bs(\I))\}\subseteq\cM_\I$ and $|\cA|<\bs(\I)$, so there is $(B_n)\in \cM_\I$ such that for every  $\alpha<\cf(\bs(\I))$ we have  $A_n^\alpha \subseteq B_n$ for all but finitely many $n\in \omega$.
Consequently, for every $(E_n)\in \cE$ we have $E_n\subseteq B_n$ for all but finitely many $n\in \omega$, a contradiction with the choice of the family $\cE$.

Finally, we show the regularity of $\add_\omega(\I)$.
Suppose for sake of contradiction that 
$\add_\omega(\I) = \bigcup\{A_\alpha:\alpha<\kappa\}$ where $\kappa <\add_\omega(\I)$ and $|A_\alpha|<\add_\omega(\I)$ for every $\alpha<\kappa$.
Let $\cB\subseteq\I$ be such that $|\cB|=\add_\omega(\I)$ and for every $(D_k)\in \I^\omega$ there is $B\in \cB$ with $B\not\subseteq D_k$ for any $k<\omega$.
Then we can write $\cB=\bigcup\{\cB_\alpha:\alpha<\kappa\}$ with $|\cB_\alpha|=|A_\alpha|$ for every $\alpha<\kappa$.
Since $|\cB_\alpha|<\add_\omega(\I)$ and $\cB_\alpha\subseteq\I$ for every $\alpha<\kappa$, we can find $(C^\alpha_n)\in \I^\omega$ such that for every $B\in \cB_\alpha$ there is $n\in \omega$ with $B\subseteq C^\alpha_n$. 
Let $\cC=\{C^\alpha_n:\alpha<\kappa,n<\omega\}$.
Then $\cC\subseteq\I$ and $|\cC|\leq \kappa\cdot\omega<\add_\omega(\I)$ (by Proposition \ref{prop:bounds-for-add-omega}(\ref{prop:bounds-for-add-omega:item:included-geq-omega-one})), so there is $(D_k)\in \I^\omega$
such that for every $\alpha<\kappa$ and $n<\omega$ there is $k<\omega$ with $C^\alpha_n\subseteq D_k$.
Thus, for every $B\in \cB$ we can find $k$ with $B\subseteq D_k$, a contradiction.
\end{proof}


\subsection{P-ideals}

An ideal $\I$ is a \emph{P-ideal} if for every countable family $\cA\subseteq\I$ there exists a set $B\in \I$ such that $A\setminus B$ is finite for every $A\in \cA$.
It is easy to see that $\adds(\I)\geq \omega_1$ for P-ideals and $\adds(\I)=\omega$ for non-P-ideals.

\begin{remark}
The inequality from Proposition~\ref{prop:bounds-for-bsigma-bs}(\ref{prop:bounds-for-bsigma-bs:item}) is interesting, in a sense, only for P-ideals.
Indeed,
by 
Proposition~\ref{prop:bounds-for-bs}(\ref{prop:bounds-for-bsigma:item:geq-omega-one})(\ref{prop:bounds-for-bs:item:equals-min-b-and-adds}), 
 we have $\bb_s(\fin,\I,\I)=\adds(\I)=\omega<\omega_1\leq\bs(\I)$
 in the case of non-P-ideals.
\end{remark}

\begin{proposition}
\label{prop:for-Pideals-add-omega-equals-adds}
If $\I$ is a P-ideal on $\omega$, then 
$$\add_\omega(\I)=\add^*(\I).$$
\end{proposition}

\begin{proof}
From Proposition~\ref{prop:bounds-for-add-omega}(\ref{prop:bounds-for-add-omega:item:add-omega-geq-add-star}) it follows that we only need to show 
$\add_\omega(\I)\leq \add^*(\I)$.
Let $\cA\subseteq\I$ be a witness for $\add^*(\I)$. We claim that $\cA$ is also a witness for $\add_\omega(\I)$.
Indeed, take any $(B_n)\in [\I]^\omega$.
Since $\I$ is a P-ideal, there is $B\in \I$ such that $|B_n\setminus B|<\omega$ for every $n\in \omega$.
Since $B\in \I$, we find $A\in \cA$ such that $A\setminus B$ is infinite.
Consequently, $A\setminus B_n$ is infinite for every $n\in \omega$. Thus, $A\not\subseteq B_n$ for any $n\in \omega$.
\end{proof}

\begin{remark}
    The cardinal $\adds(\I)$ has been extensively studied so far (see e.g.~a very good survey of Hru\v{s}\'{a}k \cite{MR2777744}). However, this cardinal is useless for non-P-ideals (because its value is $\omega$ for non-P-ideals). On the other hand, the cardinal $\add_\omega(\I)$ coincides with $\adds(\I)$ for P-ideals (as shown in Proposition~\ref{prop:for-Pideals-add-omega-equals-adds}) and it can distinguish non-P-ideals (as shown in Theorem~\ref{thm:value-of-bsigma-bs-add-omega-for-known-ideals}).  
    Thus, the cardinal $\add_\omega(\I)$ is, in a sense, more sensitive variant of $\adds(\I)$, and maybe it will turn out to be more useful than $\adds(\I)$ in the future research.
\end{remark}

\begin{corollary}
\label{cor:bsigma-for-P-ideals}
If $\I$ is a P-ideal on $\omega$ then 
$$\bs(\I)=\bb_s(\fin,\I,\I)=\min\{\bb,\adds(\I)\}\leq \add_\omega(\I).$$
\end{corollary}

\begin{proof}
It is enough to note that 
$\bs(\I)\geq \bb_s(\fin,\I,\I)=\min\{\bb,\adds(\I)\}$ 
follows from Proposition~\ref{prop:bounds-for-bsigma-bs}(\ref{prop:bounds-for-bs:item:equals-min-b-and-adds}),
 $\bs(\I)\leq\bb$ follows from Proposition~\ref{prop:bounds-for-bsigma}(\ref{prop:bounds-for-bsigma:item:leq-b}),
$\bs(\I)\leq\adds(\I)$ follows from Theorem~\ref{thm:bsigma-leq-add-omega} and 
Proposition~\ref{prop:for-Pideals-add-omega-equals-adds}
and  
$\min\{\bb,\adds(\I)\}\leq \add_\omega(\I)$ follows from Proposition~\ref{prop:for-Pideals-add-omega-equals-adds}.
\end{proof}


\subsection{Fubini products}

\begin{lemma}
\label{lem:bnumbers-for-Fubini-products-for-left-ideal}
\label{otimes-lemma1}
 Let $\I,\J$ be ideals on $\omega$.
\begin{enumerate}
    \item $\bs(\I\otimes\J)\leq\bs(\I)$.\label{lem:bnumbers-for-Fubini-products-for-left-ideal:item:bsigma}
    \item $\add_\omega(\I\otimes\J)\leq\add_\omega(\I)$.\label{lem:bnumbers-for-Fubini-products-for-left-ideal:item:add-omega}
\end{enumerate}
\end{lemma}

\begin{proof}

(\ref{lem:bnumbers-for-Fubini-products-for-left-ideal:item:bsigma})
Let $\{( E^\alpha_k) : \alpha<\bs(\I)\}\subseteq\cM_\I$ be a witness for $\bs(\I)$. Define $D^\alpha_k=E^\alpha_k\times\omega$ for all $k\in\omega$ and $\alpha<\bs(\I)$. Then $\{( D^\alpha_k) : \alpha<\bs(\I)\}\subseteq\cM_{\I\otimes\J}$. 

Fix any $( B_k) \in\cM_{\I\otimes\J}$. Define $A_k=\{n\in\omega: (B_k)_{(n)}\notin\J\}$ for all $k\in\omega$. Then $( A_k) \in\cM_{\I}$, so there is $\alpha<\bs(\I)$ such that $Z=\{k\in\omega: E^\alpha_k\not\subseteq A_k\}\notin\fin$. For each $k\in Z$, we  pick $n_k,m_k\in\omega$ such that $n_k\in E^\alpha_k\setminus A_k$ and $m_k\in\omega\setminus (B_k)_{(n_k)}$ (which is possible as $n_k\notin A_k$ implies $(B_k)_{(n_k)}\in\J$). Then $(n_k,m_k)\in D^\alpha_k\setminus B_k$ for each $k\in Z$, so $D^\alpha_k \not\subseteq  B_k$ for infinitely many $k\in\omega$.

(\ref{lem:bnumbers-for-Fubini-products-for-left-ideal:item:add-omega}) 
This is an easy modification of the proof of item (\ref{lem:bnumbers-for-Fubini-products-for-left-ideal:item:bsigma}).
\end{proof}

\begin{lemma}
\label{lem:bnumbers-for-Fubini-products-for-right-ideal}
\label{otimes-lemma2}
 Let $\I,\J$ be ideals on $\omega$.
\begin{enumerate}
    \item $\bs(\I\otimes\J)\leq\bs(\J)$.\label{lem:bnumbers-for-Fubini-products-for-right-ideal:item:bsigma}
    \item $\add_\omega(\I\otimes\J)\leq\add_\omega(\J)$.\label{lem:bnumbers-for-Fubini-products-for-right-ideal:item:add-omega}
\end{enumerate}
\end{lemma}

\begin{proof}
(\ref{lem:bnumbers-for-Fubini-products-for-right-ideal:item:bsigma})
Let $\{( E^\alpha_k) : \alpha<\bs(\J)\}\subseteq\cM_\J$ be a witness for $\bs(\J)$. Define $D^\alpha_k=\omega\times E^\alpha_k$ for all $k\in\omega$ and $\alpha<\bs(\J)$. Then $\{( D^\alpha_k) : \alpha<\bs(\J)\}\subseteq\cM_{\I\otimes\J}$. 

Fix any $( B_k) \in\cM_{\I\otimes\J}$. Define $i_k=\min\{n\in\omega: (B_k)_{(n)}\in\J\}$ and $A_k=(B_k)_{(i_k)}$ for all $k\in\omega$ (note that $i_k$ is well defined as $\{n\in\omega: (B_k)_{(n)}\notin\J\}\in\I$). 
For every $k\in \omega$, we define $C_k = \bigcup_{j\leq k}A_j$.
Then $( C_k) \in\cM_{\J}$, so there is $\alpha<\bs(\J)$ such that $Z=\{k\in\omega: E^\alpha_k\not\subseteq C_k\}\notin\fin$. 

For each $k\in Z$, we pick $m_k\in\omega$ such that $m_k\in E^\alpha_k\setminus C_k$. Then for each $k\in Z$ we have $(i_k,m_k)\in D^\alpha_k\setminus B_k$ (as $(i_k,m_k)\in B_k$ would imply $m_k\in (B_k)_{(i_k)}=A_k\subseteq C_k$), so $D^\alpha_k \not\subseteq  B_k$ for infinitely many $k\in\omega$.

(\ref{lem:bnumbers-for-Fubini-products-for-right-ideal:item:add-omega})
This is an easy modification of the proof of item (\ref{lem:bnumbers-for-Fubini-products-for-right-ideal:item:bsigma}).
\end{proof}

\begin{lemma}
\label{lem:bnumbers-for-Fubini-products-for-left-right-ideal}
\label{otimes-lemma3}
 $\bs(\I\otimes\J)\geq\min \{ \bs(\I), \bs(\J) \}$
 for every ideals $\I,\J$ on $\omega$.
\end{lemma}

\begin{proof}
Suppose that $\kappa<\min(\bs(\I), \bs(\J))$ and fix any $\{( E^\alpha_k:k\in \omega) : \alpha<\kappa\}\subseteq\cM_{\I\otimes\J}$. We want to define $( A_k) \in\cM_{\I\otimes\J}$ such that for each $\alpha<\kappa$ we have $E^\alpha_k\not\subseteq A_k$ only for finitely many $k\in\omega$.

For each $\alpha<\kappa$ and $k,n\in\omega$ put: $$D^\alpha_k=\{m\in\omega: (E^\alpha_k)_{(m)}\notin\J\},$$
$$C^\alpha_{k,n}=
\begin{cases}
(E^\alpha_k)_{(n)}, & \text{if }n\in \omega\setminus D^\alpha_k,\\
\emptyset, & \text{otherwise.}\\
\end{cases}$$
Then $\{( D^\alpha_k) : \alpha<\kappa\}\subseteq\cM_\I$. Since $\kappa<\bs(\I)$, there is $( B_k) \in\cM_\I$ such that for each $\alpha<\kappa$ we have $\{k\in\omega: D^\alpha_k\not\subseteq B_k\}\in\fin$. Moreover, for each $n\in\omega$ the family $\{( \bigcup_{i\leq k} C^\alpha_{i,n}:k\in \omega) : \alpha<\kappa\}\subseteq\cM_\J$, so there is $( B^n_k) \in\cM_\J$ such that $\{k\in\omega: \bigcup_{i\leq k} C^\alpha_{i,n}\not\subseteq B^n_k\}\in\fin$ for each $\alpha<\kappa$ (as $\kappa<\bs(\J)$). 

For every $\alpha<\kappa$ define $f_\alpha\in\omega^\omega$ by:
$$f_\alpha(n)=\max\left\{k\in\omega: \bigcup_{i\leq k} C^\alpha_{i,n}\not\subseteq B^n_k\right\}.$$
By Proposition~\ref{prop:bounds-for-bsigma}(\ref{prop:bounds-for-bsigma:item:geq-omega-one}), $\kappa<\bb$, so there is $g\in\omega^\omega$ such that $f_\alpha+1\leq^\star g$ for all $\alpha<\kappa$. 

Define: 
$$A_k=(B_k\times\omega)\cup\bigcup_{n\in\omega}\left(\{n\}\times\left(B^n_k\cup B^n_{g(n)}\right)\right).$$

Fix $\alpha<\kappa$. We want to find $m\in\omega$ such that $E^\alpha_k\subseteq A_k$ for each $k>m$. Define $n_0=\max\{n\in\omega: f_\alpha(n)+1>g(n)\}$ ($n_0$ is well defined as $f_\alpha+1\leq^\star g$) and:
$$m=\max\left(\{n_0\}\cup\{f_\alpha(n): n\leq n_0\}\cup\{k\in\omega: D^\alpha_k\not\subseteq B_k\}\right)$$
($m$ is well defined as $\{k\in\omega: D^\alpha_k\not\subseteq B_k\}\in\fin$). 

Fix $k>m$ and any $(x,y)\in E^\alpha_k$. We will show that $(x,y)\in A_k$. There are four possible cases:
\begin{itemize}
    \item if $x\in D^\alpha_k$ then $x\in B_k$ (as $k>m\geq\max\{k'\in\omega: D^\alpha_{k'}\not\subseteq B_{k'}\}$), so $(x,y)\in B_k\times\omega\subseteq A_k$;
    \item if $x\notin D^\alpha_k$ and $f_\alpha(x)<k$ then $(x,y)\in E^\alpha_k$ implies $y\in(E^\alpha_k)_{(x)}=C^\alpha_{k,x}\subseteq\bigcup_{i\leq k}C^\alpha_{i,x}\subseteq B^x_k$, so $(x,y)\in\{x\}\times B^x_k\subseteq A_k$;
    \item if $x\notin D^\alpha_k$ and $x\leq n_0$ then  $k>m\geq\max\{f_\alpha(n): n\leq n_0\}\geq f_\alpha(x)$, so this case is covered by the previous one;
    \item if $x\notin D^\alpha_k$, $f_\alpha(x)\geq k$ and $x>n_0$ then $k\leq f_\alpha(x)<g(x)$ (by $x>n_0$), so $(x,y)\in E^\alpha_k$ implies $y\in(E^\alpha_k)_{(x)}=C^\alpha_{k,x}\subseteq \bigcup_{i\leq g(x)}C^\alpha_{i,x}\subseteq B^x_{g(x)}$ (as $g(x)>f_\alpha(x)$), so $(x,y)\in\{x\}\times B^x_{g(x)}\subseteq A_k$.
\end{itemize}
This finishes the entire proof.
\end{proof}

\begin{theorem}
\label{otimes}
\label{thm:bnumbers-for-Fubini-products}
 Let $\I,\J$ be ideals on $\omega$.
\begin{enumerate}
\item $\bnumber_s(\I\otimes\J) = \bnumber_s(\I)$.\label{thm:bnumbers-for-Fubini-products:item:bs}
\item $\bs(\I\otimes\J)=\min(\bs(\I), \bs(\J))$.\label{thm:bnumbers-for-Fubini-products:item:bsigma}
\item 
$\add_\omega(\I\otimes\J)\leq\min\{\add_\omega(\I), \add_\omega(\J)\}$\label{thm:bnumbers-for-Fubini-products:item:add-omega}
\end{enumerate}
\end{theorem}

\begin{proof}
(\ref{thm:bnumbers-for-Fubini-products:item:bs})
See \cite[Theorem~5.13]{MR4472525}.

(\ref{thm:bnumbers-for-Fubini-products:item:bsigma}) and (\ref{thm:bnumbers-for-Fubini-products:item:add-omega})
It follows from Lemmas \ref{otimes-lemma1}, \ref{otimes-lemma2} and \ref{otimes-lemma3}.
\end{proof}

The following example shows that, in general, there is no way to calculate 
$\add_\omega(\I\otimes\J)$ using only values $\add_\omega(\I)$ and $\add_\omega(\J)$.

\begin{example}
\label{exm:add-omega-for-FINxFIN}
    $\add_\omega(\Fin\otimes\Fin) =  \bnumber$, but $\add_\omega(\Fin)=\infty$.
\end{example}

\begin{proof}
The equality $\add_\omega(\Fin)=\infty$ follows from Proposition \ref{prop:bounds-for-bs}(\ref{prop:bounds-for-add-omega:item:add-omega-finite}) as $\Fin$ is countably generated.

Now, we  show  $\add_\omega(\Fin\otimes\Fin)\leq \bnumber$.
Let $\{f_\alpha:\alpha<\bnumber\}$ be an $\leq^*$-unbounded set in $\omega^\omega$.
For each $\alpha$, we define $A_\alpha=\{(n,k)\in \omega^2: k\leq f_\alpha(n)\}$.
Then $\{A_\alpha:\alpha<\bnumber\}\subseteq \Fin\otimes\Fin$, and we claim that for every $(B_n)\in (\Fin\otimes\Fin)^\omega$ there is $\alpha$ with $A_\alpha\not\subseteq B_n$.
Indeed, take any 
$(B_n)\in (\Fin\otimes\Fin)^\omega$ and suppose, for sake of contradiction, that for every  $\alpha$ there is $n\in \omega$ with $A_\alpha\subseteq B_n$.
Since $B_n\in \Fin\otimes\Fin$, for every $n\in \omega$ there is $g_n\in \omega^\omega$ and $k_n\in \omega$ with $\max( (B_n)_{(k)})\leq g_n(k)$ for every $k\geq k_n$.
Let $g\in \omega^\omega$ be such that $g_n\leq^* g$ for every $n\in \omega$ (we can find $g$ because $\bnumber\geq \omega_1$).
Consequently, $f_\alpha\leq^*g$ for every $\alpha<\bnumber$, a contradiction.

Finally, we show that $\add_\omega(\Fin\otimes\Fin)\geq \bnumber$.
Let $\cA\subseteq\Fin\otimes\Fin$ with $|\cA|<\bnumber$.
If we find $(B_n)\in (\Fin\otimes\Fin)^\omega$ such that for every $A\in \cA$ there is $n\in \omega$ with $A\subseteq B_n$, then $\add(\Fin\otimes\Fin,\omega)\geq\bnumber$, and the proof will be finished.

For every $A\in \cA$ there is $f_A\in \omega^\omega$ and $n_A\in \omega$ such that $\max(A_{(n)})\leq f_A(n)$ for every $n\geq n_A$.
Since $|\cA|<\bnumber$, there is $g\in \omega^\omega$ such that $f_A\leq^* g$ for every $A\in \cA$. Hence, for each $A\in \cA$ there is $k_A\in\omega$ such that $f_A(n)\leq g(n)$ for all $n>k_A$. 

For every $n\in \omega$, we define $B_n=(n\times\omega)\cup\{(i,k)\in \omega^2: k\leq g(i)\}$.
Then $B_n\in\Fin\otimes \Fin$ and $A\subseteq B_{\max(n_A,k_A)}$ for every $A\in \cA$.
\end{proof}


\subsection{Some examples and comparisons}

Denote by $\cN$ the $\sigma$-ideal of Lebesgue null subsets of $\mathbb{R}$ and recall the definition of \emph{additivity} of $\cN$:
$$\add(\cN)=\min\left\{|\mathcal{A}|: \mathcal{A}\subseteq\cN\ \wedge\ \bigcup\mathcal{A}\notin\cN\right\}.$$
It is known that $\omega_1\leq\add(\cN)\leq \bnumber\leq\cc$ (see e.g.~\cite{MR2768685}).

\begin{theorem}\ 
\label{thm:value-of-bsigma-bs-add-omega-for-known-ideals}
\begin{enumerate}

    \item $\bs(\Fin) =\bnumber_s(\Fin) =\bnumber<\infty=\add_\omega(\Fin)$.\label{thm:value-of-bsigma-bs-add-omega-for-known-ideals:item:FIN}

    \item $\bs(\Fin\otimes\{\emptyset\}) =\bnumber_s(\Fin\otimes\{\emptyset\}) =\bnumber<\infty=\add_\omega(\Fin\otimes\{\emptyset\})$.\label{thm:value-of-bsigma-bs-add-omega-for-known-ideals:item:FINxEMPTY}

    \item $\bs(\I_d)= \add_\omega(\I_d)  =\add(\cN)\leq \bnumber=\bnumber_s(\I_d)$.\label{thm:value-of-bsigma-bs-add-omega-for-known-ideals:item:density}

    \item $\bs(\I_{1/n})=\add_\omega(\I_{1/n})=\add(\cN)\leq \bnumber =\bnumber_s(\I_{1/n})$.\label{thm:value-of-bsigma-bs-add-omega-for-known-ideals:item:summable}

    \item $\bs(\fin\otimes\fin)=\bnumber_s(\fin\otimes\fin)=\add_\omega(\fin\otimes\fin)=\bb$.\label{thm:value-of-bsigma-bs-add-omega-for-known-ideals:item:FINxFIN} 

   \item $\bs(\{\emptyset\}\otimes\fin)=\bnumber_s(\{\emptyset\}\otimes\fin)= \add_\omega(\{\emptyset\}\otimes\fin) =\bb$.\label{thm:value-of-bsigma-bs-add-omega-for-known-ideals:item:EMPTYxFIN} 
   
    \item $\bs(\cS)=\bnumber_s(\cS)=\add_\omega(\cS)=\omega_1$.\label{thm:value-of-bsigma-bs-add-omega-for-known-ideals:item:Solecki}    
\end{enumerate}
\end{theorem}

\begin{proof}
(\ref{thm:value-of-bsigma-bs-add-omega-for-known-ideals:item:FIN})
It follows from 
Proposition \ref{prop:bounds-for-bsigma}(\ref{prop:bounds-for-bsigma:item:FIN})
and \ref{prop:bounds-for-bs}(\ref{prop:bounds-for-bs:item:equals-b}) and Example~\ref{exm:add-omega-for-FINxFIN}.

(\ref{thm:value-of-bsigma-bs-add-omega-for-known-ideals:item:FINxEMPTY}) The equality $\add_\omega(\Fin\otimes\{\emptyset\})=\infty$ follows from Proposition \ref{prop:bounds-for-bs}(\ref{prop:bounds-for-add-omega:item:add-omega-finite}) as $\Fin\otimes\{\emptyset\}$ is countably generated. The equality $\bb_s(\Fin\otimes\{\emptyset\})=\bb$ follows from \cite[Example 5.15]{MR4472525} and $\bs(\Fin\otimes\{\emptyset\})=\bb$ follows from Theorem~\ref{thm:bsigma-leq-bs}.

(\ref{thm:value-of-bsigma-bs-add-omega-for-known-ideals:item:density}) and (\ref{thm:value-of-bsigma-bs-add-omega-for-known-ideals:item:summable}) 
    It is known that $\adds(\I_d)=\adds(\I_{1/n})=\add(\cN)$ (see e.g.~\cite{MR2777744})
    and $\bnumber_s(\I_d)=\bnumber_s(\I_{1/n})=\bnumber$ 
    (see \cite[Corollary~6.4]{MR4472525}).
Thus, the remaining inequalities follow from Proposition~\ref{prop:for-Pideals-add-omega-equals-adds}
and Corollary~\ref{cor:bsigma-for-P-ideals}

(\ref{thm:value-of-bsigma-bs-add-omega-for-known-ideals:item:FINxFIN})
It follows from item (\ref{thm:value-of-bsigma-bs-add-omega-for-known-ideals:item:FIN}), Theorem~\ref{thm:bnumbers-for-Fubini-products}(\ref{thm:bnumbers-for-Fubini-products:item:bs})(\ref{thm:bnumbers-for-Fubini-products:item:bsigma}) and Example~\ref{exm:add-omega-for-FINxFIN}.

(\ref{thm:value-of-bsigma-bs-add-omega-for-known-ideals:item:EMPTYxFIN})
It is known that $\adds(\{\emptyset\}\otimes\fin)=\bb$ (see e.g.~\cite{MR2777744})
and $\bnumber_s(\{\emptyset\}\otimes\Fin)=\bnumber$ (see \cite[Theorem~5.13]{MR4472525}).
Thus, the remaining inequalities follow from Proposition~\ref{prop:for-Pideals-add-omega-equals-adds}
and Corollary~\ref{cor:bsigma-for-P-ideals}

    (\ref{thm:value-of-bsigma-bs-add-omega-for-known-ideals:item:Solecki})
It is known that $\bnumber_s(\cS)=\omega_1$ (see \cite[Theorem 7.4]{MR4472525}).
Then, using Proposition~\ref{prop:bounds-for-bsigma}(\ref{prop:bounds-for-bsigma:item:geq-omega-one}) and Theorem~\ref{thm:bsigma-leq-bs}, we obtain 
$\bs(\cS)=\omega_1$. Below we show that $\add_\omega(\cS)=\omega_1$.

Let $Y\subseteq 2^\omega$ be any set of cardinality $\omega_1$. We claim that $\cA=\{G_y: y\in Y\}$, where $G_y=\{A\in\Omega: y\in A\}$, witnesses $\add_\omega(\cS)=\omega_1$. Let $(  B_n) \in\I^\omega$. Then for each $n\in\omega$ there are $k_n\in\omega$ and $x^n_0,\ldots,x^n_{k_n}\in 2^\omega$ such that $B_n\subseteq \bigcup_{i\leq k_n}G_{x^n_i}$. Since $|Y|=\omega_1$, we can find $y\in Y\setminus\{x^n_i: n\in\omega,i\leq k_n\}$. We will show that $G_y\not\subseteq B_n$ for all $n$.

Let $n\in\omega$. There is $k\in\omega$ such that $2^k>2k_n$ and $y\restriction k\neq x^n_i\restriction k$ for all $i\leq k_n$. Since $2^k>2k_n$, we can find pairwise distinct $y_j\in 2^k$, for $j<2^{k-1}-1$, such that $y\restriction k\neq y_j$ and $x^n_i\restriction k\neq y_j$ for all $i\leq k_n$. Then 
$$X=\{x\in 2^\omega: x\restriction k=y\restriction k\text{ or }x\restriction k=y_j\text{ for some }j<2^{k-1}-1\}\in\Omega$$
and $X\in G_y\setminus B_n$.
\end{proof}

By Theorem~\ref{thm:bsigma-leq-bs} we know that $\bs(\I)=\min\{\bb_s(\I),\add_\omega(\I)\}$ for every ideal $\I$. The above result shows that 
$$\bs(\I)=\bb_s(\I)<\add_\omega(\I)$$ 
for some P-ideal (item (\ref{thm:value-of-bsigma-bs-add-omega-for-known-ideals:item:FIN})) as well as for some non-P-ideal (item (\ref{thm:value-of-bsigma-bs-add-omega-for-known-ideals:item:FINxEMPTY})). Since $\add(\cN)<\bb$ is consistent (see e.g.~\cite{MR2768685}), we obtain that it is consistent that 
$$\bs(\I)=\add_\omega(\I)<\bb_s(\I)$$ 
for some P-ideals (items (\ref{thm:value-of-bsigma-bs-add-omega-for-known-ideals:item:density}) and (\ref{thm:value-of-bsigma-bs-add-omega-for-known-ideals:item:summable})). Next example shows that the latter is consistent also for some non-P-ideal.

\begin{example}
Consider the ideal $\I=\fin\otimes\cS$, which is not a P-ideal. By Theorems~\ref{otimes} and \ref{thm:value-of-bsigma-bs-add-omega-for-known-ideals} and Corollary~\ref{cor:bounds-for-bsigma} we have $\bs(\I)=\bs(\cS)=\omega_1$ and $\add_\omega(\I)=\omega_1$. On the other hand, $\bb_s(\I)=\bb_s(\fin)=\bb$ (by \cite[Theorems 4.2 and 5.13]{MR4472525}). It is known that $\omega_1<\bb$ is consistent (see e.g.~\cite{MR2768685}). Thus, consistently $\bs(\I)=\add_\omega(\I)<\bb_s(\I)$ also for non-P-ideals.
\end{example}


\section{Spaces not distinguishing  convergence can be of arbitrary cardinality}
\label{sec:spaces-of-arbitrary-cardinality}

In this section, we show (see e.g.~Corollary~\ref{cor:spaces-of-arbitrary-cardinality-may-distinguis-convergences}) that the properties ``$X\in (\text{$\I$-p,$\I$-$\sigma$-u})$'' 
``$X\in(\text{$\I$-p,$\I$-qn})$''
and 
``$X\in(\text{$\I$-qn,$\I$-$\sigma$-u})$''
are of the topological nature rather than set-theoretic.

\begin{lemma}
\label{lem:continuous-function-is-constant-everywhere-but-less-than-bs-many-points}
    Let $\I,\J$ be ideals on $\omega$ such that $\I\subseteq \J$.
Let $X$ be a topological space such that for each $f\in \cC(X)$ there is a set $Y\subseteq X$ such that $|Y|<\bs(\I,\J)$ and $f\restriction (X\setminus Y)$ is constant.
Then 
$$f_n\xrightarrow{\text{$\I$-p}}0 \implies f_n\xrightarrow{\text{$\J$-$\sigma$-u}}0
  \text{  for any sequence $(  f_n) $ in $\cC(X)$,}$$
\end{lemma}

\begin{proof}
    Let $(f_n)$ be a sequence in $\cC(X)$ such that $f_n\xrightarrow{\text{$\I$-p}}0$.
For each $n\in\omega$ there is a set $Y_n\subseteq X$ such that $|Y_n|<\bs(\I,\J)$ and  $f_n\restriction (X\setminus Y_n)$ is constant. 
Let $Y = \bigcup\{Y_n:n\in\omega\}$ and put $Z=X\setminus Y$.

Since 
$f_n\xrightarrow{\text{$\I$-p}}0$
and $\I\subseteq\J$, we have 
$f_n\xrightarrow{\text{$\J$-p}}0$.

Since $f_n\restriction Z$ are constant for each $n$ and  
$f_n\restriction Z\xrightarrow{\text{$\J$-p}}0$, we obtain 
$f_n\restriction Z\xrightarrow{\text{$\J$-u}}0$.

Since $\bs(\I,\J)$ has uncountable cofinality (by Proposition~\ref{prop:bounds-for-bsigma}(\ref{prop:bounds-for-bsigma:item:bsigma-has-uncountable-cofinality})), 
we obtain 
 $|Y|<\bs(\I,\J)$.
 Thus, we can use Theorem~\ref{thm:pointwise-implies-sigma-uniform:necessary-condition} to obtain 
$f_n\restriction Y\xrightarrow{\text{$\J$-$\sigma$-u}}0$.

Since $X=Y\cup Z$, we obtain $f_n\xrightarrow{\text{$\J$-$\sigma$-u}}0$.
\end{proof}

\begin{lemma}
\label{lem:point-with-nghds-that-are-of-cardinality-co-but-less-than-bs-many-points}
    Let $\I,\J$ be ideals on $\omega$ such that $\I\subseteq \J$.
Let $X$ be a topological space such that there exists a point $p\in X$ with the property that $|X\setminus N|<\bs(\I,\J)$ for each neighborhood $N$ of $p$.
Then 
$$f_n\xrightarrow{\text{$\I$-p}}0 \implies f_n\xrightarrow{\text{$\J$-$\sigma$-u}}0
  \text{  for any sequence $(  f_n) $ in $\cC(X)$,}$$
\end{lemma}

\begin{proof}
    Let $(f_n)$ be a sequence in $\cC(X)$ such that $f_n\xrightarrow{\text{$\I$-p}}0$.
We will show that we can apply Lemma~\ref{lem:continuous-function-is-constant-everywhere-but-less-than-bs-many-points} to the space $X$. 
Let $f:X\to\mathbb{R}$ be continuous. 
Using  continuity of $f$ only at the point $p$, for each $n\in \omega$ we find a neighborhood $N_n$ of $p$ such that 
$|f(p)-f(x)|<1/n$
for each $x\in N_n$.
Let $Y = X\setminus \bigcap\{N_n:n\in \omega\}$. 
Since $\bs(\I,\J)$ has uncountable cofinality (by Proposition~\ref{prop:bounds-for-bsigma:item}(\ref{prop:bounds-for-bsigma:item:bsigma-has-uncountable-cofinality})), 
we obtain $|Y|<\bs(\I,\J)$.
 Then 
$|f(p)-f(x)|<1/n$ for each $x\in X\setminus Y$ and each $n\in \omega$.
Consequently, $f\restriction (X\setminus Y)$ is constant with the value $f(p)$.
\end{proof}

The following theorem shows that one cannot strengthen Theorem~\ref{thm:pointwise-implies-sigma-uniform:sufficient-condition-for-discrete-space} to all normal spaces.

\begin{theorem}
\label{thm:pointwise-implies-sigma-uniform:sufficient-condition-for-discrete-space:NOT-true-for-normal-spaces}
Let $\I,\J$ be ideals on $\omega$ such that $\I\subseteq \J$.
There exists a Hausdorff compact (hence normal) space $X$ of arbitrary cardinality such that 
$$f_n\xrightarrow{\text{$\I$-p}}0 \implies f_n\xrightarrow{\text{$\J$-$\sigma$-u}}0
  \text{  for any sequence $(  f_n) $ in $\cC(X)$.}$$
\end{theorem}

\begin{proof}
Obviously every finite space $X$ has the required property. Let $D$ be an infinite (of arbitrary cardinality) discrete spaces. Then $D$ is a Hausdorff and locally compact space but not a compact space. Thus, the Alexandroff one-point compactification $X=D\cup\{\infty\}$ of $D$ is a Hausdorff compact space. In particular, $X$ is a normal space (see e.g.~\cite[Theorem~3.1.9]{MR1039321}).

We will show that we can apply Lemma~\ref{lem:point-with-nghds-that-are-of-cardinality-co-but-less-than-bs-many-points} to the space $X$. 
Recall that open neighborhoods of the point $\infty$ are of the form  $N=(D\setminus K)\cup \{\infty\}$ where $K$ is a compact subset of $D$ (see e.g.~\cite[Theorem~3.5.11]{MR1039321}).
Since every compact subset of $D$ is finite, we have that $X\setminus N$  is finite for every neighborhood $N$ of the point $\infty$. In particular, $|X\setminus N|<\bs(\I,\J)$ (by Proposition~\ref{prop:bounds-for-bsigma}(\ref{prop:bounds-for-bsigma:item:geq-omega-one})).
\end{proof}

In the above theorem, all but one point are isolated in the constructed spaces. Below, we show that there also are  required spaces (at least of cardinality up to the cardinality of the continuum) in which only countably many points are isolated. 

\begin{theorem}
\label{thm:pointwise-implies-sigma-uniform:sufficient-condition-for-discrete-space:NOT-true-for-normal-spaces:countably-many-isolated-points}
Let $\I,\J$ be ideals on $\omega$ such that $\I\subseteq \J$.
There exists a Hausdorff separable, sequentially compact, compact (hence  normal) space $X$ of  arbitrary  cardinality up to $\continuum$ such that only countably many points of $X$ are isolated and  
$$f_n\xrightarrow{\text{$\I$-p}}0 \implies f_n\xrightarrow{\text{$\J$-$\sigma$-u}}0
  \text{  for any sequence $(  f_n) $ in $\cC(X)$.}$$
\end{theorem}

\begin{proof}
Obviously every finite space $X$ has the required property. 
Let $\cA$ be an infinite (of arbitrary cardinality up to $\continuum$) almost disjoint family $\cA$ of infinite subsets of $\omega$ (see e.g.~\cite[Lemma~9.21]{MR1940513}).

Let $\Psi(\cA) = \omega \cup \cA$
and introduce a topology on $\Psi(\cA)$ as follows:
 the points of $\omega$
are isolated and a basic neighborhood of $A\in\cA$ has the form $\{A\}\cup
(A\setminus F)$ with $F$ finite.

Let 
$\Phi(\cA)  = \Psi(\cA)\cup \{\infty\}$
be the Alexandroff one-point compactification of $\Psi(\cA)$. 
It is known (see e.g.~\cite{MR3822423})   
that  $\Phi(\cA)$ is Hausdorff, compact, sequentially compact and separable.

We will show that we can apply Lemma~\ref{lem:point-with-nghds-that-are-of-cardinality-co-but-less-than-bs-many-points} to the space $\Phi(\cA)$. 
Recall that open neighborhoods of the point $\infty$ are of the form  $U=(\Psi(\cA)\setminus K)\cup \{\infty\}$ where $K$ is a compact subset of $\Psi(\cA)$ (see e.g.~\cite[Theorem~3.5.11]{MR1039321}).
Since for every compact subset $K$ of $\Psi(\cA)$, 
both sets 
$K\cap \cA$ and $(K \cap \omega) \setminus \bigcup \{A : A\in K\cap \cA\}$  are finite (see e.g.~\cite{MR3822423}),
we obtain that 
$\Phi(\cA)\setminus N$  is countable for every neighborhood $N$ of the point $\infty$. In particular, $|\Phi(\cA)\setminus N|<\bs(\I,\J)$ (by Proposition~\ref{prop:bounds-for-bsigma}(\ref{prop:bounds-for-bsigma:item:geq-omega-one})).
\end{proof}

\begin{corollary}
\label{cor:spaces-of-arbitrary-cardinality-may-distinguis-convergences}
For every ideal $\I$ the classes  ($\I$-p,$\I$-$\sigma$-u), ($\I$-p,$\I$-qn) and ($\I$-qn,$\I$-$\sigma$-u)  contain spaces of arbitrary cardinality.
\end{corollary}

\begin{proof}
Let $\I$ be an ideal and $X$ be a space from Theorem~\ref{thm:pointwise-implies-sigma-uniform:sufficient-condition-for-discrete-space:NOT-true-for-normal-spaces}.
Then 
$$f_n\xrightarrow{\text{$\I$-p}}0 \implies f_n\xrightarrow{\text{$\I$-$\sigma$-u}}0
  \text{  for any sequence $(  f_n) $ in $\cC(X)$.}$$
On the other hand, by Proposition~\ref{prop:easy-implications-one-ideal} we have
$$f_n\xrightarrow{\text{$\I$-$\sigma$-u}}0 \implies f_n\xrightarrow{\text{$\I$-p}}0
  \text{  for any sequence $(  f_n) $ in $\cC(X)$.}$$
Thus, $X\in (\text{$\I$-p,$\I$-$\sigma$-u})$.  
Now, Corollary~\ref{cor:only-one-class}  implies that  
$X\in (\text{$\I$-p,$\I$-qn})$ and $X\in (\text{$\I$-qn,$\I$-$\sigma$-u})$.
\end{proof}


\subsection{Subsets of reals  not distinguishing  convergence}

Obviously, countable subspaces of $\R$ are in 
the classes  ($\I$-p,$\I$-$\sigma$-u), ($\I$-p,$\I$-qn) and ($\I$-qn,$\I$-$\sigma$-u).
Uncountable spaces constructed in the proof of Corollary~\ref{cor:spaces-of-arbitrary-cardinality-may-distinguis-convergences} are not homeomorphic to any subspace of $\R$ as those spaces contain uncountable discrete subspaces.
Below we show that consistently there is an uncountable subspace of $\R$ in the considered classes at least for the ideal  $\I=\{\emptyset\}\otimes\Fin$.

Recall that an uncountable set $S\subseteq R$ is called a \emph{Sierpinski set} if $S\cap N$ is countable for every Lebesgue null set $N\subseteq \R$. 

\begin{theorem}
\label{thm:Sierpinski-set-not-distinguishes-convergence}
    Let $\I=\{\emptyset\}\otimes\Fin$.
    \begin{enumerate}
        \item 
Every Sierpi\'{n}ski set belongs to the classes  ($\I$-p,$\I$-$\sigma$-u), ($\I$-p,$\I$-qn) and ($\I$-qn,$\I$-$\sigma$-u).\label{thm:Sierpinski-set-not-distinguishes-convergence:item}

\item Consistently (e.g.~under the Continuum Hypothesis), there exists  an uncountable subspace of $\R$ which belongs to the classes ($\I$-p,$\I$-$\sigma$-u), ($\I$-p,$\I$-qn) and ($\I$-qn,$\I$-$\sigma$-u).\label{thm:Sierpinski-set-not-distinguishes-convergence:item-CH}
    \end{enumerate}
\end{theorem}

\begin{proof}
    (\ref{thm:Sierpinski-set-not-distinguishes-convergence:item})
Let $S\subseteq\R$ be a Sierpi\'{n}ski set. Without loss of generality we can assume that $S\subseteq [0,1]$. 
By Corollary~\ref{cor:only-one-class}, it is enough to show that $S\in (\text{$\I$-p,$\I$-$\sigma$-u})$.
Let $(f_n)$ be a sequence in $\cC(S)$ which is $\I$-pointwise convergent to zero.
By \cite[Theorem~5]{MR2270601}, there is a set $A\in \I$ such that the subsequence $(f_n:n\in \omega\setminus A)$ is $\Fin$-pointwise convergent to zero.
There are a $G_\delta$ set $G\subseteq [0,1]$ and continuous functions $g_n:G\to\R$ such that $S\subseteq G$ and $f_n = g_n\restriction S$ for every $n\in \omega\setminus A$ (see e.g.~\cite[Theorem~3.8]{MR1321597}).
It is not difficult to see that the set 
$B = \{x\in G: \text{$(g_n(x):n\in \omega\setminus A) $ is $\Fin$-convergent to zero}\}$
is Borel and $S\subseteq B$.
Applying repeatedly Egorov's theorem (see e.g.~\cite[Proposition 3.1.4]{MR3098996}) to the sequence $(g_n\restriction B:n\in \omega\setminus A)$, we find a sequence of pairwise disjoint Borel sets $\{C_k:k\in \omega\}$ such that $(g_n\restriction C_k:n\in \omega\setminus A)$ is uniformly convergent to zero
and $N=B\setminus \bigcup\{C_k:k\in \omega\}$ is Lebesgue null.
Then  $S\cap N$ is countable, so $(f_n\restriction (S\cap N): n\in \omega\setminus A)$ is $\sigma$-uniformly convergent to zero.
Consequently, $(f_n: n\in \omega\setminus A)$ is $\sigma$-uniformly convergent to zero.
Since $A\in \I$, we obtain that $(f_n:n\in \omega)$ is $\I$-$\sigma$-uniformly convergent to zero.

(\ref{thm:Sierpinski-set-not-distinguishes-convergence:item-CH})
    It follows from item~(\ref{thm:Sierpinski-set-not-distinguishes-convergence:item}) as under the Continuum Hypothesis there is a Sierpi\'{n}ski set (see e.g.~\cite[Theorem~2.2]{MR0776624}).
\end{proof}

\begin{question}
Let $\I$ be an arbitrary ideal. Do the classes  ($\I$-p,$\I$-$\sigma$-u), ($\I$-p,$\I$-qn) and ($\I$-qn,$\I$-$\sigma$-u) contain an uncountable subspace of $\R$?
\end{question}


\section{Bounding numbers of binary relations}
\label{sec:b-of-relations}

If $R$ is a binary relation, then by $\dom(R)$ and $\ran(R)$ we denote the domain and range of $R$, respectively, i.e.
$\dom(R)=\{x: \exists y\, ((x,y)\in R)\}$ and $\ran(R) = \{y: \exists x\, ((x,y)\in R)\}$.
A set $B\subseteq \dom(R)$ is called \emph{$R$-unbounded} if for every $y\in \ran(R)$ there is $x\in B$ with $(x,y)\notin R$.
Following Vojt\'{a}\v{s} \cite{MR1234291}, for a binary relation $R$ we define 
$$\bnumber(R) = \min\{|B|: \text{$B$ is  an $R$-unbounded set}\}.$$

It is easy to see that 
the bounding  number $\bnumber$ is equal to the bounding number of the relation
$\leq^*$ on $\omega^\omega$ i.e.
$\bnumber = \bnumber(\leq^*)$.

\begin{definition}\ 
\begin{enumerate}

\item The binary relation  $\succeq$ is define by  $\dom(\succeq) =\ran(\succeq) =  \omega^\omega$ and 
$$x\succeq y \iff \left\{m\in\omega: \exists k\in\omega \, (x(k)\leq m< y(k))\right\}\in\fin.$$

\item The  binary relation 
$\leq^\omega$ is defined by 
$\dom(\leq^\omega) = 2^\omega$, 
$\ran(\leq^\omega) = (2^\omega)^\omega$ and 
$$x\leq^\omega (y_k) \iff \exists k\in \omega\, \forall n\in \omega (x(n)\leq y_k(n)).$$

\item 
For an ideal  $\I$ on $\omega$, 
 the binary relation $\leq_\I$ is defined by
$\dom(\leq_\I) =\omega^\omega$, $\ran(\leq_\I) = \omega^\omega$ and  
$$x\leq_\I y \iff \{n\in \omega: x(n)>y(n)\}\in \I.$$ 
In a similar manner we define $<_\I$, $\geq_\I$ and $>_\I$.

\end{enumerate}

\end{definition}

\begin{proposition}
The relation  $\succeq$ is a preorder on $\omega^\omega$ i.e.~the relation $\succeq$ is reflexive and transitive. 
\end{proposition}
\begin{proof}
Since reflexivity is obvious, we show only transitivity.
If $f\succeq g$ and $g\succeq h$, then put:
$n=\max 
(\{m\in\omega: \exists k\in\omega \, (f(k)\leq m< g(k))\}
\cup 
\{m\in\omega: \exists k\in\omega\, (g(k)\leq m< h(k))\}).$
Fix any $m>n$. Then for each $k\in\omega$, if $m<h(k)$ then also $m< g(k)$, and consequently $m< f(k)$. Hence, 
$\{m\in\omega: \exists k\in\omega\, (f(k)\leq m< h(k))\}\subseteq \{i\in \omega:i\leq n\}\in\fin.$
\end{proof}

\begin{notation}
For an ideal $\I$, we define
\begin{equation*}
    \begin{split}
\cC_\I & = \{x\in 2^\omega: x^{-1}[\{1\}]\in \I\} = \{\chf_A:A\in \I\},
\\
\cD_\I
&=
\{x\in\omega^\omega: x^{-1}[\{n\}]\in\I \text{ for every $n\in \omega$}\}.
    \end{split}
\end{equation*}
\end{notation}

\begin{theorem}
\label{thm:bs-ideal-as-Vojtas-b}
\label{thm:b-sigma-as-Vojtas-b}
\label{thm:add-omega-as-Vojtas-b}
\label{bIII=bVojtas}
Let $\I,\J,\K$ be ideals on $\omega$.

\begin{enumerate}

\item $\bs(\I,\J)=\bb(\succeq\cap(\cD_\I\times\cD_\J))$.\label{thm:b-sigma-as-Vojtas-b:item}

\item $\add_\omega(\I,\J)=\bnumber(\leq^\omega\cap(\cC_\I\times (\cC_\J)^\omega))$.\label{thm:add-omega-as-Vojtas-b:item} 

\item \label{thm:bs-ideal-as-Vojtas-b:item}
$\bb_s(\I,\J,\K)=\bb(\geq_\I\cap(\cD_\K\times\cD_\J))$.
If $\J\cap \K\subseteq\I$, then  $\bb_s(\I,\J,\K)=\bb(>_\I\cap(\cD_\K\times\cD_\J))$.

\end{enumerate}
\end{theorem}

\begin{proof}
(\ref{thm:b-sigma-as-Vojtas-b:item})
First, we show $\bs(\I,\J)\leq \bb(\succeq\cap(\cD_\I\times\cD_\J))$. 
Let $\{f_\alpha: \alpha<\bb(\succeq\cap(\cD_\I\times\cD_\J))\}$ be unbounded in $(\succeq\cap(\cD_\I\times\cD_\J))$. Define $E^\alpha_k=f_\alpha^{-1}[[0,k]]$ for each $k\in\omega$ and $\alpha<\bb(\succeq\cap(\cD_\I\times\cD_\J))$. Then $\cE=\{(  E^\alpha_k) : \alpha<\bb(\succeq\cap(\cD_\I\times\cD_\J))\}\subseteq\cM_\I$ as each $f_\alpha$ is in $\cD_\I$. We claim that $\cE$ witnesses $\bs(\I,\J)$.

Fix $(  A_k) \in\cM_\J$ and define $B_k=(A_k\cup\{k\})\setminus\bigcup_{i<k}B_i$. Then $(  B_k) $ is a partition of $\omega$ into sets belonging to $\J$. Define a function $g\in\omega^\omega$ by 
$$g(n)=k\ \Leftrightarrow\ n\in B_k.$$
Then $g\in\cM_\J$, so there is $\alpha<\bb(\succeq\cap(\cD_\I\times\cD_\J))$ such that $f_\alpha\not\succeq g$. Hence, there are infinitely many $m\in\omega$ such that $f_\alpha(n_m)\leq m<g(n_m)$ for some $n_m\in\omega$. Observe that in this case we have $n_m\in E^\alpha_m$ and $n_m\notin A_m$ (as $n_m\in A_m$ would imply $n_m\in \bigcup_{i\leq m}B_i$ and consequently $g(n_m)\leq m$).

Second, we show $\bs(\I,\J)\geq \bb(\succeq\cap(\cD_\I\times\cD_\J))$. 
Let $\{(  E^\alpha_k) : \alpha<\bs(\I,\J)\}\subseteq\cM_\I$ be a witness for $\bs(\I,\J)$. For each $\alpha<\bs(\I,\J)$ define $f_\alpha\in\omega^\omega$ by:
$$f_\alpha(n)=k\ \Leftrightarrow\ n\in B^\alpha_k,$$
where $B^\alpha_k=(E^\alpha_k\cup\{k\})\setminus \bigcup_{i<k}B^\alpha_i$. Note that each $f_\alpha$ is well defined and belongs to $\cD_\I$ as $(  B^\alpha_k) $ is a partition of $\omega$ into sets belonging to $\I$. We claim that $\{f_\alpha: \alpha<\bb(\I,\J)\}$ is unbounded in $(\succeq\cap(\cD_\I\times\cD_\J))$.

Fix any $g\in\cD_\J$ and define $A_k=g^{-1}[[0,k]]$. Then $(  A_k) \in\cM_\J$, so there is $\alpha<\bs(\I,\J)$ such that $E^\alpha_k\not\subseteq A_k$ for infinitely many $k\in\omega$. Note that if $n\in E^\alpha_k\setminus A_k$ for some $k\in\omega$, then $f_\alpha(n)\leq k$ (as $n\in E^\alpha_k\subseteq\bigcup_{i\leq k} B^\alpha_i$) and $k<g(n)$. Thus, there are infinitely many $k\in\omega$ such that $f_\alpha(n)\leq k<g(n)$ for some $n\in\omega$.

(\ref{thm:add-omega-as-Vojtas-b:item})
It easily follows from the fact that  $A\subseteq B \iff \chf_A(n)\leq \chf_B(n)$ for every $n\in \omega$.

(\ref{thm:bs-ideal-as-Vojtas-b:item}) See \cite[Theorem~3.10]{MR4472525}.
\end{proof}


\section{Subsets of reals  distinguishing  convergence}
\label{sec:subsets-of-R-distinguishing}

In this section, we  show (Theorem~\ref{thm:subset-of-R-not-distinguishing-convergence}) that, in a sense, the connection between cardinals $\bs(\I)$ ($\bnumber_s(\I)$, $\add_\omega(\I)$, resp.) and  
$\non(\text{$\I$-p,$\I$-$\sigma$-u})$ ($\non(\text{$\I$-p,$\I$-qn})$, $\non(\text{$\I$-qn,$\I$-$\sigma$-u})$, resp.)
is even deeper than that following from  the proof of Corollary \ref{cor:pointwise-versus-sigma-uniform:non}, as here we  obtain  subspaces of $\R$ as spaces which realize the minimum value of spaces not distinguishing the considered  convergences.

\begin{lemma}
\label{lemma-space}
\label{lem:subset-of-R-not-distinguishing-convergence}
Let $\I,\J,\K$ be ideals on $\omega$.
\begin{enumerate}

\item \label{lem:subset-of-R-not-distinguishing-convergence:pointwise-vs-sigma-uniform}
For each $n\in\omega$, let $f_n:\omega^\omega\to\mathbb{R}$ be given  by $f_n(x) = \frac{1}{x(n)+1},$
for all $x\in\omega^\omega$.
Then 
\begin{enumerate}
    \item $\forall x\in \omega^\omega\, (f_n(x)\xrightarrow{\I}0 \iff x\in\cD_\I)$,\label{lem:subset-of-R-not-distinguishing-convergence:pointwise-vs-sigma-uniform:pointwise}
    \item $\forall X\subseteq\cD_\I\, (f_n\restriction X\xrightarrow{\text{$\K$-$\sigma$-u}}0 \iff X \text{ is bounded in } (\succeq\cap(\cD_\I\times\cD_\K))$.\label{lem:subset-of-R-not-distinguishing-convergence:pointwise-vs-sigma-uniform:sigma-uniform}
\end{enumerate}

\item \label{lem:subset-of-R-not-distinguishing-convergence:qn-vs-sigma-uniform}
For each $n\in \omega$, we define $g_n: 2^\omega\to\R$ by 
$g_n(x)=x(n)$ for all $x\in 2^\omega$.
Then 
\begin{enumerate}
    \item $\forall X\subseteq 2^\omega\,( g_n\restriction X \xrightarrow{\text{$\J$-qn}}0 \iff X\subseteq \cC_\J)$,\label{lem:subset-of-R-not-distinguishing-convergence:qn-vs-sigma-uniform:qn} 
    \item $\forall X\subseteq\cC_\J\, (g_n\restriction X\xrightarrow{\text{$\K$-$\sigma$-u}}0 \iff X \text{ is bounded in } (\leq^\omega\cap(\cC_\J\times(\cC_\K)^\omega))$.\label{lem:subset-of-R-not-distinguishing-convergence:qn-vs-sigma-uniform:sigma-uniform}
\end{enumerate}

\item \label{lem:subset-of-R-not-distinguishing-convergence:pointwise-vs-qn}
For each $n\in \omega$, we define $h_n: \omega^\omega \to \R$ by 
$h_n(x)=\frac{1}{x(n)+1}$ for all $x\in \omega^\omega$.
Then 
\begin{enumerate}
    \item $\forall x\in \omega^\omega\, (h_n(x)\xrightarrow{\I}0 \iff x\in \cD_\I)$,\label{lem:subset-of-R-not-distinguishing-convergence:pointwise-vs-qn:pointwise} 
    \item $\forall X\subseteq \cD_\I \, (h_n\restriction X\xrightarrow{\text{$\J$-qn}}0 \iff X \text{ is bounded in } (\geq_\J\cap (\cD_\I\times \cD_\J)))$.\label{lem:subset-of-R-not-distinguishing-convergence:pointwise-vs-qn:qn}
\end{enumerate}

\end{enumerate}
\end{lemma}

\begin{proof}
(\ref{lem:subset-of-R-not-distinguishing-convergence:pointwise-vs-sigma-uniform:pointwise}) 
If $x\in\cD_\I$ and $\varepsilon>0$ then find $k\in\omega$ such that $\varepsilon\geq\frac{1}{k+1}$ and observe that: 
$$\left\{n\in\omega: f_n(x)\geq\varepsilon\right\}\subseteq\left\{n\in\omega: \frac{1}{x(n)+1}\geq\frac{1}{k+1}\right\}=x^{-1}[[0,k]]\in\I.$$ 
On the other hand, if $x\notin\cD_\I$ then there is $k\in\omega$ such that $x^{-1}[\{k\}]\notin\I$. Then $\left\{n\in\omega: f_n(x)\geq\frac{1}{k+1}\right\}=x^{-1}[[0,k]]\supseteq x^{-1}[\{k\}]\notin\I$.

(\ref{lem:subset-of-R-not-distinguishing-convergence:pointwise-vs-sigma-uniform:sigma-uniform})
If $X\subseteq\cD_\I$ is bounded in $(\succeq\cap(\cD_\I\times\cD_\K))$ by some $g\in\cD_\K$ then for each $x\in X$ denote $m_x=\max\left\{m\in\omega: \exists_{k\in\omega}\, x(k)\leq m< g(k)\right\}$ (recall that this set is finite since $x\succeq g$). Define $X_m=\{x\in X: m_x=m\}$ for each $m\in\omega$. Then $X=\bigcup_{m\in\omega}X_m$. We claim that $f_n\restriction X_m\xrightarrow{\text{$\K$-u}}0$ for each $m\in\omega$. 

Fix $m\in\omega$ and $\varepsilon>0$. Find $k\in\omega$ such that $\varepsilon\geq\frac{1}{k+1}$. Since $g\in\cD_\K$, $g^{-1}[0,\max\{m+1,k\}]\in\K$. Fix $n\in\omega\setminus g^{-1}[0,\max\{m+1,k\}]$ and $x\in X_m$. Then $g(n)>m+1$, so $x(n)\geq g(n)$ (otherwise we would have $x(n)\leq g(n)-1<g(n)$ which contradicts the choice of $m_x$ as $g(n)-1>m=m_x$). Thus, we have:
$$\varepsilon\geq\frac{1}{k+1}>\frac{1}{g(n)+1}\geq\frac{1}{x(n)+1}=f_n(x)$$
(as $g(n)>k$).

Assume now that $X\subseteq\cD_\I$ is unbounded in $(\succeq\cap(\cD_\I\times\cD_\K))$. Suppose to the contrary that $X=\bigcup_{m\in\omega} X_m$ for some sets $X_m$ such that $f_n\restriction X_m\xrightarrow{\text{$\K$-u}}0$ for each $m\in\omega$. 

Then for each $m,k\in\omega$ we can find $A^m_k\in\K$ such that $f_n(x)<\frac{1}{k+1}$ for all $n\in\omega\setminus A^m_k$ and $x\in X_m$. Define $A_k=\bigcup_{i\leq k}A^i_k$ (observe that if $n\in\omega\setminus A_k$ and $x\in \bigcup_{i\leq k}X_i$ then $f_n(x)<\frac{1}{k+1}$). Define $B_k=(A_k\cup\{k\})\setminus\bigcup_{i<k}B_i$, for all $k\in\omega$, and $g\in\cD_\K$ by:
$$g(n)=k\ \Leftrightarrow\ n\in B_k$$
($g$ is well defined as $( B_k) \in\cP_\K$).

Since $X$ is unbounded, there is $x\in X$ such that $x\not\succeq g$. Let $m\in\omega$ be such that $x\in X_m$. Then there is $m'>m$ such that $x(n)\leq m'< g(n)$ for some $n\in\omega$. Since $m'<g(n)$, $n\notin A_{m'}$, so $f_n(x)<\frac{1}{m'+1}$ (by $x\in X_m\subseteq\bigcup_{i\leq m'}X_i$). On the other hand, $f_n(x)=\frac{1}{x(n)+1}\geq\frac{1}{m'+1}$, since $x(n)\leq m'$. Thus, we obtained a contradiction, which proves that 
$f_n\restriction X\xrightarrow{\text{$\K$-$\sigma$-u}}0$ does not hold.

(\ref{lem:subset-of-R-not-distinguishing-convergence:qn-vs-sigma-uniform:qn}, $\implies$)
Let $X\subseteq 2^\omega$ be such that  
$g_n\restriction X \xrightarrow{\text{$\J$-qn}}0$.
Then there exists a $\J$-convergent to zero sequence  $(\varepsilon_n)$  of positive reals such that 
$\{n\in \omega:|g_n(x)|\geq \varepsilon_n\}\in \J$ for every $x\in X$.
Let $A=\{n\in \omega: \varepsilon_n>1/2\}$. Then $A\in \J$ and 
$
\{n\in \omega:x(n)=1\} = 
\{n\in \omega:|g_n(x)|>1/2\} \subseteq  
\{n\in \omega:|g_n(x)|\geq \varepsilon_n\} \cup A\in\J$ for every $x\in X$.
Thus, $x\in \cC_\J$ for every $x\in X$, and consequently $X\subseteq \cC_\J$.

(\ref{lem:subset-of-R-not-distinguishing-convergence:qn-vs-sigma-uniform:qn}, $\impliedby$)
Let $X\subseteq \cC_\J$.
We claim that any sequence $(\varepsilon_n)$ of positive reals which $\J$-converges to zero witnesses that  $g_n\restriction X \xrightarrow{\text{$\J$-qn}}0$.
Indeed, take any 
sequence $(\varepsilon_n)$ of positive reals which $\J$-converges to zero
and fix $x\in X$.
Then $A=\{n\in \omega:\varepsilon_n>1/2\}\in \J$ and 
$
\{n\in \omega:|g_n(x)|\geq \varepsilon_n\}
=
\{n\in \omega:x(n)\geq \varepsilon_n\}
\subseteq 
\{n\in \omega:x(n)\geq 1/2\} \cup \{n\in \omega: \varepsilon_n  >1/2\} =x^{-1}[\{1\}]\cup A\in \J
$.

(\ref{lem:subset-of-R-not-distinguishing-convergence:qn-vs-sigma-uniform:sigma-uniform}, $\implies$)
Let $X\subseteq \cC_\J$ and assume that 
$f_n\restriction X\xrightarrow{\text{$\K$-$\sigma$-u}}0$.
Then there exists a cover $\{X_k:k\in \omega\}$ of $X$ such that 
$f_n\restriction X_k\xrightarrow{\text{$\K$-u}}0$
for every $k\in \omega$.
For every $k\in \omega$, we define $A_k=\{n\in \omega:\exists x\in X_k\,(|g_n(x)|>1/2)\}$
and
$y_k=\chf_{A_k}$.
Since $A_k\in \K$ for every $k\in \omega$, we have $(y_k)\in (\cC_\K)^\omega$. 
If we show  that $x\leq^\omega (y_k)$ for every $x\in X$, the proof will be finished.
Take any $x\in X$. Then there is  $k\in \omega$ with $x\in X_k$.
If $n\in A_k$, then $x(n)\leq 1 = y_k(n)$, and if $n\in \omega\setminus A_k$, then $x(n)=g_n(x)\leq 1/2$, so $x(n)=0$ and consequently $x(n)=0\leq y_k(n)$.
All in all, $x\leq^\omega(y_k)$.

(\ref{lem:subset-of-R-not-distinguishing-convergence:qn-vs-sigma-uniform:sigma-uniform}, $\impliedby$)
Let $X\subseteq \cC_\J$ be bounded in $(\leq^\omega\cap(\cC_\J\times(\cC_\K)^\omega))$.
Then there is $(y_k)\in (\cC_\K)^\omega$ such that 
for every $x\in X$ there is $k\in \omega$ with $x(n)\leq y_k(n)$ for every $n\in \omega$.
For every $k\in \omega$, we define $X_k = \{x\in X: x(n)\leq y_k(n) \text{ for every } n\in \omega\}$.
Then $\{X_k:k\in \omega\}$ is a cover of $X$.
If we show that 
$g_n\restriction X_k \xrightarrow{\text{$\K$-u}}0$ for every $k\in \omega$, the 
proof will be finished.
Take any $k\in \omega$ and $\varepsilon>0$.
Then 
$
\{n\in \omega: \exists x\in X_k\,(|g_n(x)|\geq \varepsilon)\}
=
\{n\in \omega: \exists x\in X_k\,(x(n)\geq \varepsilon)\}
\subseteq 
\{n\in \omega: y_k(n)\geq \varepsilon)\}
\subseteq y_k^{-1}[\{1\}] 
\in \K$.

(\ref{lem:subset-of-R-not-distinguishing-convergence:pointwise-vs-qn:pointwise}) This is item (\ref{lem:subset-of-R-not-distinguishing-convergence:pointwise-vs-sigma-uniform:pointwise}) as $f_n=h_n$ for all $n\in\omega$.

(\ref{lem:subset-of-R-not-distinguishing-convergence:pointwise-vs-qn:qn}, $\implies$)
Let $X\subseteq\cD_\I$ be such that 
$h_n\restriction X\xrightarrow{\text{$\J$-qn}}0$.
Then there exists a $\J$-convergent to zero sequence $(\varepsilon_n)$ of positive reals such that 
$\{n\in\omega: |h_n(x)|\geq \varepsilon_n\}\in \J$ for every $x\in X$.
We define $y\in \omega^\omega$ by 
$y(n) = \max\{0,[1/\varepsilon_n-1]\}$ for every $n\in \omega$ (here $[r]$ means the integer part of $x$).
We claim that $y\in \cD_\J$ and $y$ is  a $\geq_\J$-bound of a set $X$.

To see that $y\in \cD_\J$, we fix $k\in \omega$ and notice 
$\{n\in \omega: y(n)\leq k\} 
=
\{n\in \omega: 1/\varepsilon_n-1<k+1\}
=
\{n\in \omega: \varepsilon_n>1/(k+2)\}\in \J
$
as $(\varepsilon_n)$ is $\J$-convergent to zero. 

To see that  $y$ is  a $\geq_\J$-bound of a set $X$, we fix $x\in X$ and notice
$
\{n\in \omega: x(n)<y(n)\} 
\subseteq  
\{n\in \omega: x(n)<1/\varepsilon_n-1\}
=
\{n\in \omega: \frac{1}{x(n)+1}>\varepsilon_n\}
=
\{n\in \omega: |h_n(x)|>\varepsilon_n\}\in \J
$
as the sequence $(\varepsilon_n)$ witnesses $h_n\restriction X\xrightarrow{\text{$\J$-qn}}0$.

(\ref{lem:subset-of-R-not-distinguishing-convergence:pointwise-vs-qn:qn}, $\impliedby$)
Let $X\subseteq\cD_\I$ be $\geq_\J$-bounded in 
$(\geq_\J\cap (\cD_\I\times \cD_\J))$.
Then there exists $y\in \cD_\J$ such that $\{n\in\omega: x(n)<y(n)\}\in\J$ for every $x\in X$.
We define a sequence $(\varepsilon_n)$ by 
$\varepsilon_n = 1/(y(n)+1)$ for every $n\in \omega$. 
We claim that $(\varepsilon_n)$   is a witness for 
$h_n\restriction X\xrightarrow{\text{$\J$-qn}}0$

To see that $(\varepsilon_n)$ is $\J$-convergent to zero, we fix $\varepsilon>0$
and notice 
$
\{n\in \omega: \varepsilon_n\geq \varepsilon\}
=
\{n\in \omega: y(n)\leq 1/\varepsilon-1\}\in \J
$
as $y\in \cD_\J$.

Now, we fix $x\in X$ and notice that 
$
\{n\in \omega: |h_n(x)|\geq \varepsilon_n\}
=
\{n\in \omega: x(n)\leq 1/\varepsilon_n-1\}
\subseteq 
\{n\in \omega: x(n)< y(n)\}
\cup
\{n\in \omega: x(n)\leq 1/\varepsilon_n-1\land x(n)\geq y(n)\}
\subseteq 
\{n\in \omega: x(n)< y(n)\}
\cup
\{n\in \omega: y(n)\leq 1/\varepsilon_n-1\}
\in \J$ as $y\in \cD_\J$.
\end{proof}

\begin{theorem}
\label{thm:subset-of-R-not-distinguishing-convergence}
Let $\I$ be an ideal on $\omega$.
\begin{enumerate}

\item 
There is   $X\subseteq \omega^\omega$ such that $|X|= \non(\text{$\I$-p,$\I$-$\sigma$-u})$
and 
$X\notin (\text{$\I$-p,$\I$-$\sigma$-u})$.\label{thm:subset-of-R-not-distinguishing-convergence:pointwise-vs-sigma-uniform}

\item 
If $\I$ is not countably generated then there is   $X\subseteq 2^\omega$ such that $|X|= \non(\text{$\I$-qn,$\I$-$\sigma$-u})$
and 
$X\notin (\text{$\I$-qn,$\I$-$\sigma$-u})$.\label{thm:subset-of-R-not-distinguishing-convergence:qn-vs-sigma-uniform}

\item 
There is  $X\subseteq \omega^\omega$ such that $|X|= \non(\text{$\I$-p,$\I$-qn})$
and 
$X\notin (\text{$\I$-p,$\I$-qn})$.\label{thm:subset-of-R-not-distinguishing-convergence:pointwise-vs-qn}

\end{enumerate}
\end{theorem}

\begin{proof}
(\ref{thm:subset-of-R-not-distinguishing-convergence:pointwise-vs-sigma-uniform})
Since $\bs(\I)=\bb(\succeq\cap(\cD_\I\times\cD_\I))<\infty$
(by Theorem~\ref{thm:b-sigma-as-Vojtas-b}(\ref{thm:b-sigma-as-Vojtas-b:item}) and Proposition~\ref{prop:bounds-for-bsigma}(\ref{prop:bounds-for-bsigma:item:leq-b})),
there is a set $X\subseteq \cD_\I$ 
which is  unbounded in 
$\succeq\cap(\cD_\I\times\cD_\I)$ and $|X|= \bs(\I)$.
By Corollary~\ref{cor:pointwise-versus-sigma-uniform:non}(\ref{cor:pointwise-versus-sigma-uniform:non-equals-bsigma}), 
$|X|=\non(\text{$\I$-p,$\I$-$\sigma$-u})$
and
by Lemma~\ref{lem:subset-of-R-not-distinguishing-convergence}(\ref{lem:subset-of-R-not-distinguishing-convergence:pointwise-vs-sigma-uniform}) we obtain 
$X\notin (\text{$\I$-p,$\I$-$\sigma$-u})$.

(\ref{thm:subset-of-R-not-distinguishing-convergence:qn-vs-sigma-uniform})
Since 
$\add_\omega(\I)=\bnumber(\leq^\omega\cap(\cC_\I\times (\cC_\J)^\omega))<\infty$
(by Theorem~\ref{thm:add-omega-as-Vojtas-b}(\ref{thm:add-omega-as-Vojtas-b:item}) and Proposition~\ref{prop:bounds-for-add-omega}(\ref{prop:bounds-for-add-omega:item:add-omega-finite})),
there is a set 
$X\subseteq \cC_\I$ 
which is  unbounded in 
$(\leq^\omega\cap(\cC_\I\times (\cC_\J)^\omega))$ and 
$|X|= \add_\omega(\I)$.
By Corollary~\ref{cor:quasinormal-versus-sigma-uniform:non}(\ref{cor:quasinormal-versus-sigma-uniform:non-equals-add-omega}), 
$|X|=\non(\text{$\I$-qn,$\I$-$\sigma$-u})$
and
by Lemma~\ref{lem:subset-of-R-not-distinguishing-convergence}(\ref{lem:subset-of-R-not-distinguishing-convergence:qn-vs-sigma-uniform}) we obtain 
$X\notin (\text{$\I$-qn,$\I$-$\sigma$-u})$.

(\ref{thm:subset-of-R-not-distinguishing-convergence:pointwise-vs-qn})
Since 
$\bnumber_s(\I)=\bb(\succeq\cap(\cD_\I\times\cD_\I))<\infty$
(by Theorem~\ref{thm:bs-ideal-as-Vojtas-b}(\ref{thm:bs-ideal-as-Vojtas-b:item}) and Proposition~\ref{prop:bounds-for-bs}(\ref{prop:bounds-for-bs:item:bs-is-leq-continuum})),
there is a set $X\subseteq \cD_\I$ 
which is  unbounded in 
$\geq_\J\cap(\cD_\I\times\cD_\J)$ and $|X|= \bnumber_s(\I)$.
By Corollary~\ref{cor:pointwise-versus-quasinormal:non}(\ref{cor:pointwise-versus-quasinormal:non-equals-bs}), 
$|X|=\non(\text{$\I$-p,$\I$-qn})$
and
by Lemma~\ref{lem:subset-of-R-not-distinguishing-convergence}(\ref{lem:subset-of-R-not-distinguishing-convergence:pointwise-vs-qn}) we obtain 
$X\notin (\text{$\I$-p,$\I$-qn})$.
\end{proof}

\begin{remark}
    Since  $\omega^\omega$ is homeomorphic with $\R\setminus\Q$ 
and $2^\omega$ is homeomorphic with the Cantor ternary subset of $\R$
(see e.g.~\cite{MR1321597}), we can write ``$X\subseteq\R$'' instead of ``$X\subseteq\omega^\omega$'' and ``$X\subseteq 2^\omega$'' in Theorem~\ref{thm:subset-of-R-not-distinguishing-convergence}.
\end{remark}

\begin{remark}
We know that  $\non(\text{$\I$-p,$\I$-$\sigma$-u}) = \bs(\I)\leq \bnumber$ (by Corollary~\ref{cor:pointwise-versus-sigma-uniform:non} and Proposition~\ref{prop:bounds-for-bsigma}(\ref{prop:bounds-for-bsigma:item:leq-b})) and it is known that $\bnumber<\continuum$ is consistent (see e.g.~\cite{MR2768685}). Consequently, a subset of the reals which distinguishes the considered convergences and  constructed in the proof of Theorem~\ref{thm:subset-of-R-not-distinguishing-convergence} can have the cardinality strictly less than the cardinality of the continuum. On the other hand, the whole set $\cD_\I$ is a subset of reals of cardinality continuum which distinguishes between $\I$-pointwise and $\I$-$\sigma$-uniform convergences (by Lemma~\ref{lem:subset-of-R-not-distinguishing-convergence}(\ref{lem:subset-of-R-not-distinguishing-convergence:pointwise-vs-sigma-uniform}) as $\cD_\I$ is unbounded in $\succeq\cap(\cD_\I\times\cD_\I)$). Similar reasoning can be performed in the case of the classes ($\I$-qn,$\I$-$\sigma$-u) (provided that $\I$ is not countably generated) and ($\I$-p,$\I$-qn).
\end{remark}


\section{Distinguishing between spaces not distinguishing convergences}
\label{sec:Distinguishing-between-spaces}

If $\bs(\J)<\bs(\I)$, then using Corollary~\ref{cor:pointwise-versus-sigma-uniform:non}(\ref{cor:pointwise-versus-sigma-uniform:non-equals-bsigma}) we see that there exists  a space
$X\in (\text{$\I$-p,$\I$-$\sigma$-u})$ such that $X\notin (\text{$\J$-p,$\J$-$\sigma$-u})$, and using Theorem~\ref{thm:subset-of-R-not-distinguishing-convergence}(\ref{thm:subset-of-R-not-distinguishing-convergence:pointwise-vs-sigma-uniform}), one can even find $X\subseteq \R$ with the above property (and similarly for other types of considered convergences). As an application of this method we have: 
\begin{proposition}\ 
\label{prop:distinguishing-between-spaces-not-distinguishing-convergence}

\begin{enumerate}
\item 
The following statments are consistent with ZFC.
\begin{enumerate}

    \item There is $X\subseteq\R$ such that  
$X\in (\text{$\fin$-p,$\fin$-$\sigma$-u})$ and $X\notin (\text{$\I_d$-p,$\I_d$-$\sigma$-u})$.\label{prop:distinguishing-between-spaces-not-distinguishing-convergence:item:p-vs-sigma-u}

    \item There is $X\subseteq\R$ such that  
$X\in (\text{$\fin$-p,$\fin$-qn})$ and $X\notin (\text{$\cS$-p,$\cS$-qn})$.\label{prop:distinguishing-between-spaces-not-distinguishing-convergence:item:p-vs-qn}

\end{enumerate}    

\item There is $X\subseteq\R$ such that  
$X\in (\text{$\fin$-qn,$\fin$-$\sigma$-u})$ and $X\notin (\text{$\I_d$-qn,$\I_d$-$\sigma$-u})$.\label{prop:distinguishing-between-spaces-not-distinguishing-convergence:item:qn-vs-sigma-u}

\end{enumerate}    
\end{proposition}

\begin{proof}
(\ref{prop:distinguishing-between-spaces-not-distinguishing-convergence:item:p-vs-sigma-u})
By Theorem~\ref{thm:value-of-bsigma-bs-add-omega-for-known-ideals}, 
we have $\bs(\fin)=\bnumber$ and $\bs(\I_d)=\add(\cN)$
and it is known (see e.g.~\cite{MR2768685}) that 
$\add(\cN)<\bb$ is consistent with ZFC.

(\ref{prop:distinguishing-between-spaces-not-distinguishing-convergence:item:p-vs-qn})
By Theorem~\ref{thm:value-of-bsigma-bs-add-omega-for-known-ideals}, 
we have $\bs(\fin)=\bnumber$ and $\bs(\cS)=\omega_1$
and it is known (see e.g.~\cite{MR2768685}) that 
$\omega_1<\bb$ is consistent with ZFC.

(\ref{prop:distinguishing-between-spaces-not-distinguishing-convergence:item:qn-vs-sigma-u})
By Theorem~\ref{thm:value-of-bsigma-bs-add-omega-for-known-ideals}, 
we have $\add_\omega(\fin)=\infty> \add(\cN) = \add_\omega(\I_d)$.
\end{proof}

However, if $\bs(\J)=\bb$ (so it has  the largest possible value, as shown in Proposition~\ref{prop:bounds-for-bsigma}(\ref{prop:bounds-for-bsigma:item:leq-b})), then the above described method is useless
for distinguishing between spaces not distinguishing considered convergences. In particular, this is the case for $\J=\fin$ (by Proposition~\ref{prop:bounds-for-bsigma}(\ref{prop:bounds-for-bsigma:item:FIN})).

\begin{question}
Do there exist a space $X$ and an ideal $\I$ such that 
$X\in (\text{$\I$-p,$\I$-$\sigma$-u})$ but  $X\notin (\text{$\fin$-p,$\fin$-$\sigma$-u})$?\end{question}


\bibliographystyle{amsplain}
\bibliography{paper}

\providecommand{\bysame}{\leavevmode\hbox to3em{\hrulefill}\thinspace}
\providecommand{\MR}{\relax\ifhmode\unskip\space\fi MR }
\providecommand{\MRhref}[2]{%
  \href{http://www.ams.org/mathscinet-getitem?mr=#1}{#2}
}
\providecommand{\href}[2]{#2}
\begin{thebibliography}{10}

\bibitem{MR2768685}
Andreas Blass, \emph{Combinatorial cardinal characteristics of the continuum},
  Handbook of set theory. {V}ols. 1, 2, 3, Springer, Dordrecht, 2010,
  pp.~395--489. \MR{2768685}

\bibitem{MR1108577}
Zuzana Bukovsk\'{a}, \emph{Quasinormal convergence}, Math. Slovaca \textbf{41}
  (1991), no.~2, 137--146. \MR{1108577}

\bibitem{MR2463820}
Lev Bukovsk\'{y}, \emph{On {$wQN_*$} and {$wQN^*$} spaces}, Topology Appl.
  \textbf{156} (2008), no.~1, 24--27. \MR{2463820}

\bibitem{MR2778559}
\bysame, \emph{The structure of the real line}, Instytut Matematyczny Polskiej
  Akademii Nauk. Monografie Matematyczne (New Series) [Mathematics Institute of
  the Polish Academy of Sciences. Mathematical Monographs (New Series)],
  vol.~71, Birkh\"{a}user/Springer Basel AG, Basel, 2011. \MR{2778559}

\bibitem{MR3622377}
Lev Bukovsk\'{y}, Pratulananda Das, and Jaroslav \v{S}upina, \emph{Ideal
  quasi-normal convergence and related notions}, Colloq. Math. \textbf{146}
  (2017), no.~2, 265--281. \MR{3622377}

\bibitem{MR2294632}
Lev Bukovsk\'{y} and Jozef Hale\v{s}, \emph{{$QN$}-spaces, {$wQN$}-spaces and
  covering properties}, Topology Appl. \textbf{154} (2007), no.~4, 848--858.
  \MR{2294632}

\bibitem{MR1129696}
Lev Bukovsk\'{y}, Ireneusz Rec\l{}aw, and Miroslav Repick\'{y}, \emph{Spaces
  not distinguishing pointwise and quasinormal convergence of real functions},
  Topology Appl. \textbf{41} (1991), no.~1-2, 25--40. \MR{1129696}

\bibitem{MR1815270}
\bysame, \emph{Spaces not distinguishing convergences of real-valued
  functions}, Topology Appl. \textbf{112} (2001), no.~1, 13--40. \MR{1815270}

\bibitem{MR0924678}
Michael Canjar, \emph{Countable ultraproducts without {CH}}, Ann. Pure Appl.
  Logic \textbf{37} (1988), no.~1, 1--79. \MR{924678}

\bibitem{MR1036675}
R.~Michael Canjar, \emph{Cofinalities of countable ultraproducts: the existence
  theorem}, Notre Dame J. Formal Logic \textbf{30} (1989), no.~4, 539--542.
  \MR{1036675}

\bibitem{CanjarPhD}
Robert~Michael Canjar, \emph{Model-theoretic properties of countable
  ultraproducts without the {C}ontinuum {H}ypothesis}, ProQuest LLC, Ann Arbor,
  MI, 1982, Thesis (Ph.D.)--University of Michigan. \MR{2632174}

\bibitem{MR3098996}
Donald~L. Cohn, \emph{Measure theory}, second ed., Birkh\"{a}user Advanced
  Texts: Basler Lehrb\"{u}cher. [Birkh\"{a}user Advanced Texts: Basel
  Textbooks], Birkh\"{a}user/Springer, New York, 2013. \MR{3098996}

\bibitem{MR0515120}
\'{A}. Cs\'{a}sz\'{a}r and M.~Laczkovich, \emph{Some remarks on discrete
  {B}aire classes}, Acta Math. Acad. Sci. Hungar. \textbf{33} (1979), no.~1-2,
  51--70. \MR{515120}

\bibitem{MR3038073}
Pratulananda Das and Debraj Chandra, \emph{Spaces not distinguishing pointwise
  and {$\mathcal I$}-quasinormal convergence}, Comment. Math. Univ. Carolin.
  \textbf{54} (2013), no.~1, 83--96. \MR{3038073}

\bibitem{MR1039321}
Ryszard Engelking, \emph{General topology}, second ed., Sigma Series in Pure
  Mathematics, vol.~6, Heldermann Verlag, Berlin, 1989, Translated from the
  Polish by the author. \MR{1039321}

\bibitem{MR0048548}
H.~Fast, \emph{Sur la convergence statistique}, Colloq. Math. \textbf{2}
  (1951), 241--244 (1952). \MR{48548}

\bibitem{MR4472525}
Rafa\l{} Filip\'{o}w and Adam Kwela, \emph{Yet another ideal version of the
  bounding number}, J. Symb. Log. \textbf{87} (2022), no.~3, 1065--1092.
  \MR{4472525}

\bibitem{MR3179991}
Rafa\l{} Filip\'{o}w and Marcin Staniszewski, \emph{On ideal equal
  convergence}, Cent. Eur. J. Math. \textbf{12} (2014), no.~6, 896--910.
  \MR{3179991}

\bibitem{MR2899832}
Rafa\l{} Filip\'{o}w and Piotr Szuca, \emph{Three kinds of convergence and the
  associated $\mathcal{I}$-{B}aire classes}, J. Math. Anal. Appl. \textbf{391}
  (2012), no.~1, 1--9. \MR{2899832}

\bibitem{MR0816582}
J.~A. Fridy, \emph{On statistical convergence}, Analysis \textbf{5} (1985),
  no.~4, 301--313. \MR{816582}

\bibitem{MR3822423}
F.~Hern\'{a}ndez-Hern\'{a}ndez and M.~Hru\v{s}\'{a}k, \emph{Topology of
  {M}r\'{o}wka-{I}sbell spaces}, Pseudocompact topological spaces, Dev. Math.,
  vol.~55, Springer, Cham, 2018, pp.~253--289. \MR{3822423}

\bibitem{MR0776620}
R.~Hodel, \emph{Cardinal functions. {I}}, Handbook of set-theoretic topology,
  North-Holland, Amsterdam, 1984, pp.~1--61. \MR{776620}

\bibitem{MR2777744}
Michael Hru\v{s}\'{a}k, \emph{Combinatorics of filters and ideals}, Set theory
  and its applications, Contemp. Math., vol. 533, Amer. Math. Soc., Providence,
  RI, 2011, pp.~29--69. \MR{2777744}

\bibitem{MR2270601}
Jakub Jasi\'{n}ski and Ireneusz Rec\l{}aw, \emph{Ideal convergence of
  continuous functions}, Topology Appl. \textbf{153} (2006), no.~18,
  3511--3518. \MR{2270601}

\bibitem{MR1940513}
Thomas Jech, \emph{Set theory}, Springer Monographs in Mathematics,
  Springer-Verlag, Berlin, 2003, The third millennium edition, revised and
  expanded. \MR{1940513}

\bibitem{MR1321597}
Alexander~S. Kechris, \emph{Classical descriptive set theory}, Graduate Texts
  in Mathematics, vol. 156, Springer-Verlag, New York, 1995. \MR{1321597}

\bibitem{MR3784399}
Adam Kwela, \emph{Ideal weak {QN}-spaces}, Topology Appl. \textbf{240} (2018),
  98--115. \MR{3784399}

\bibitem{MR0776624}
Arnold~W. Miller, \emph{Special subsets of the real line}, Handbook of
  set-theoretic topology, North-Holland, Amsterdam, 1984, pp.~201--233.
  \MR{776624}

\bibitem{MR1477547}
Ireneusz Rec\l{}aw, \emph{Metric spaces not distinguishing pointwise and
  quasinormal convergence of real functions}, Bull. Polish Acad. Sci. Math.
  \textbf{45} (1997), no.~3, 287--289. \MR{1477547}

\bibitem{MR1800160}
Miroslav Repick\'{y}, \emph{Spaces not distinguishing convergences}, Comment.
  Math. Univ. Carolin. \textbf{41} (2000), no.~4, 829--842. \MR{1800160}

\bibitem{MR4336563}
\bysame, \emph{Spaces not distinguishing ideal convergences of real-valued
  functions}, Real Anal. Exchange \textbf{46} (2021), no.~2, 367--394.
  \MR{4336563}

\bibitem{MR4336564}
\bysame, \emph{Spaces not distinguishing ideal convergences of real-valued
  functions, {II}}, Real Anal. Exchange \textbf{46} (2021), no.~2, 395--421.
  \MR{4336564}

\bibitem{MR2280899}
Masami Sakai, \emph{The sequence selection properties of {$C_p(X)$}}, Topology
  Appl. \textbf{154} (2007), no.~3, 552--560. \MR{2280899}

\bibitem{MR3624786}
Marcin Staniszewski, \emph{On ideal equal convergence {II}}, J. Math. Anal.
  Appl. \textbf{451} (2017), no.~2, 1179--1197. \MR{3624786}

\bibitem{Steinhaus}
Hugo Steinhaus, \emph{Sur la convergence ordinaire et la convergence
  asymptotique}, Colloq. Math. \textbf{2} (1949), no.~1, 73--74.

\bibitem{MR2881299}
Boaz Tsaban and Lyubomyr Zdomskyy, \emph{Hereditarily {H}urewicz spaces and
  {A}rhangel'ski\u{\i} sheaf amalgamations}, J. Eur. Math. Soc. (JEMS)
  \textbf{14} (2012), no.~2, 353--372. \MR{2881299}

\bibitem{MR1234291}
Peter Vojt\'{a}\v{s}, \emph{Generalized {G}alois-{T}ukey-connections between
  explicit relations on classical objects of real analysis}, Set theory of the
  reals ({R}amat {G}an, 1991), Israel Math. Conf. Proc., vol.~6, Bar-Ilan
  Univ., Ramat Gan, 1993, pp.~619--643. \MR{1234291}

\bibitem{MR3924519}
Viera \v{S}ottov\'{a} and Jaroslav \v{S}upina, \emph{Principle
  {${\text{S}}_1(\mathcal{P},\mathcal{R})$}: ideals and functions}, Topology
  Appl. \textbf{258} (2019), 282--304. \MR{3924519}

\bibitem{MR3423409}
Jaroslav \v{S}upina, \emph{Ideal {QN}-spaces}, J. Math. Anal. Appl.
  \textbf{435} (2016), no.~1, 477--491. \MR{3423409}

\end{thebibliography}

\end{document}